\documentclass[final,1p,times]{elsarticle}
\usepackage{amsmath,amsthm,amssymb}
\usepackage{color}
\usepackage{hyperref}
\usepackage{siunitx}
\def\R{\mathbb R}
\def\aaa{\mathbf{a}}
\def\bbb{\mathbf{b}}

\def\eee{\mathbf{e}}
\def\fff{\mathbf{f}}

\def\hhh{\mathbf{h}}

\def\jjj{\mathbf{j}}

\def\mmm{\mathbf{m}}
\def\nnn{\mathbf{n}}

\def\sss{\mathbf{s}}

\def\vvv{\mathbf{v}}
\def\www{\mathbf{w}}
\def\xxx{\mathbf{x}}

\def\zzz{\mathbf{z}}
\def\CCC{\mathbf{C}}

\def\HHH{\mathbf{H}}

\def\JJJ{\mathbf{J}}

\def\LLL{\mathbf{L}}
\def\MMM{\mathbf{M}}

\def\SSS{\mathbf{S}}

\def\eeta{\boldsymbol\eta}
\def\pphi{\boldsymbol\phi}
\def\ppi{\boldsymbol\pi}
\def\ppsi{\boldsymbol\psi}
\def\vphi{\boldsymbol\varphi}
\def\zzeta{\boldsymbol\zeta}

\def\SSSt{\frac{\partial\SSS}{\partial t}}
\def\ssst{\partial_t\sss}
\def\mmmt{\partial_t\mmm}
\def\hmmmh{\Pi_h\mmm_h}
\def\mmmh{\mmm_h}
\def\mmmhk{\mmm_{hk}}
\def\vvvhk{\vvv_{hk}}
\def\ssshk{\sss_{hk}}
\def\sssh{\sss_h}
\def\vvvh{\vvv_h}
\def\ppih{\ppi_h}
\def\MMMt{\frac{\partial\MMM}{\partial t}}
\def\heff{\mathbf{h}_{\mathrm{eff}}}
\def\HEFF{\mathbf{H}_{\mathrm{eff}}}
\def\Ce{C_{\mathrm{exch}}}
\def\Cani{C_{\mathrm{ani}}}
\def\Ms{M_s} % saturation magnetization
\def\dt{d_t}
\def\eps{\varepsilon}
\def\zero{\mathbf{0}}
\def\wtd{\widetilde}
\def\ppih{{\boldsymbol\pi}_h}
\def\EE{\mathcal E}

\def\II{\mathcal I}
\def\KK{\mathcal K}

\def\MM{\mathcal M}
\def\NN{\mathcal N}

\def\PP{\mathcal P}

\def\SS{\mathcal S}
\def\TT{\mathcal T}
\def\UU{\mathcal U}
\newcommand{\bigO}[1]{\ensuremath{\mathop{}\mathopen{}\mathcal{O}\mathopen{}\left(#1\right)}}

\def\diam{{\rm diam}}
\newcommand{\norm}[3][]{#1\left\Vert #2#1\right\Vert_{#3}}
\newcommand{\abs}[2][]{#1\left\vert #2#1\right\vert}

\def\set#1#2{\left\{#1\,:\,#2\right\}}
\def\weakto{\rightharpoonup}
\def\subto{\stackrel{sub}{-\!\!\!\to}}
\def\subweakto{\stackrel{sub}{-\!\!\!\rightharpoonup}}
\newtheorem{theorem}{Theorem}
\newtheorem{proposition}[theorem]{Proposition}
\newtheorem{lemma}[theorem]{Lemma}
\newtheorem{corollary}[theorem]{Corollary}
\newtheorem{algorithm}[theorem]{Algorithm}
\newtheorem{definition}[theorem]{Definition}
\newtheorem{remark}[theorem]{Remark}
\journal{Computers \&\ Mathematics with Applications}
\begin{document}
%%%%%%%%%%%%%%%%%%%%%%%%%%%%%%%%%%%%%%%%%%%%%%%%%%%%%%%%%%%%%%%%%%%%%%%%%%%%%%%%%%%%%%%%%%%%%%%%%%%%
\begin{frontmatter}
\title{Spin-polarized transport in ferromagnetic multilayers:\\ An unconditionally convergent FEM integrator}
%%%%%%%%%%%%%%%%%%%%%%%%%%%%%%%%%%%%%%%%%%%%%%%%%%%%%%%%%%%%%%%%%%%%%%%%%%%%%%%%%%%%%%%%%%%%%%%%%%%%
\author[issp]{Claas~Abert}
\ead{claas.abert@tuwien.ac.at}
\author[exeter]{Gino~Hrkac}
\ead{G.Hrkac@exeter.ac.uk}
\author[asc]{Marcus~Page}
\author[asc]{Dirk~Praetorius}
\ead{dirk.praetorius@tuwien.ac.at}
\author[asc]{Michele~Ruggeri\corref{cor1}}
\ead{michele.ruggeri@tuwien.ac.at}
\author[issp]{Dieter~Suess}
\ead{dieter.suess@tuwien.ac.at}
%%%%%%%%%%%%%%%%%%%%%%%%%%%%%%%%%%%%%%%%%%%%%%%%%%%%%%%%%%%%%%%%%%%%%%%%%%%%%%%%%%%%%%%%%%%%%%%%%%%%
\address[issp]{Institute of Solid State Physics, Vienna University of Technology, Austria}
\address[exeter]{College of Engineering, Mathematics and Physical Sciences, University of Exeter, United Kingdom}
\address[asc]{Institute for Analysis and Scientific Computing, Vienna University of Technology, Austria}
%%%%%%%%%%%%%%%%%%%%%%%%%%%%%%%%%%%%%%%%%%%%%%%%%%%%%%%%%%%%%%%%%%%%%%%%%%%%%%%%%%%%%%%%%%%%%%%%%%%%
\cortext[cor1]{Corresponding author}
\begin{abstract}
We propose and analyze a decoupled time-marching scheme for the coupling of the Landau-Lifshitz-Gilbert equation with a quasilinear diffusion equation for the spin accumulation.
This model describes the interplay of magnetization and electron spin accumulation in magnetic and nonmagnetic multilayer structures.
Despite the strong nonlinearity of the overall PDE system, the proposed integrator requires only the solution of two linear systems per time-step.
Unconditional convergence of the integrator towards weak solutions is proved.
\end{abstract}
\begin{keyword}
micromagnetics \sep Landau-Lifshitz-Gilbert equation \sep spin accumulation \sep finite element method
\MSC 35K55 \sep 65M60 \sep 65Z05
\end{keyword}
%%%%%%%%%%%%%%%%%%%%%%%%%%%%%%%%%%%%%%%%%%%%%%%%%%%%%%%%%%%%%%%%%%%%%%%%%%%%%%%%%%%%%%%%%%%%%%%%%%%%
\end{frontmatter}
\section{Introduction}
%%%%%%%%%%%%%%%%%%%%%%%%%%%%%%%%%%%%%%%%%%%%%%%%%%%%%%%%%%%%%%%%%%%%%%%%%%%%%%%%%%%%%%%%%%%%%%%%%%%%
\noindent The interaction between electric current and magnetization in magnetic nanostructure devices and the control of this interaction have been realized through the prediction of the spin-transfer torque by Slonczewski and Berger~\cite{berger1996,slonczewski1996}.
The transfer of spin angular momentum between the spin-polarized electrical current and the local magnetization has been observed in various magnetic devices, such as metallic spin-valves systems, magnetic tunnel junctions, and magnetic domain walls in permalloy nanowires~\cite{rs2008,sr2008}.
Based on these experiments, a number of technological applications have been proposed, e.g., STT-MRAMs, racetrack memories, and magnetic vortex oscillators~\cite{mkhkdccls2008,pht2008}.
\par The fundamental physics underlying these phenomena is understood as due to a spin torque that arises from the transfer of the spin angular momentum between conduction free electrons and magnetization. 
In the original works of Berger and Slonczewski~\cite{berger1996,slonczewski1996}, a homogeneous spin accumulation is assumed due to a current which flows through a first magnetic layer perpendicular to the interface into a second magnetic layer.
The spin torque effect leads to an interaction between the spin-polarized current and the magnetization in the second layer. 
For magnetic multilayers it has been shown that a proper description of the magnetoresistance is essential to take into account the interplay between successive interfaces~\cite{js1988,vf1993,skw1987}.
In order to calculate the spin torque transfer, the spin transport properties have to be calculated far beyond the interface.
\par The original model of Berger and Slonczewski has been extended by taking into account the diffusion process of the spin accumulation by Shpiro et al.\ for one-dimensional systems \cite{slz2003} and by Garc\'ia-Cervera and Wang~\cite{gcw2007a,gcw2007b} for three-dimensional systems.
There, the overall system of PDEs (SDLLG) is a quasilinear diffusion equation for the evolution of the spin accumulation coupled to the Landau-Lifshitz-Gilbert equation (LLG) for the magnetization dynamics.
Existence of global weak solutions to LLG goes back to~\cite{as1992}, while, in the same spirit, existence of global weak solutions to SDLLG is proved in~\cite{gcw2007a}.
\par The reliable numerical integration of LLG (and, in particular, SDLLG) faces several challenges due to the nonuniqueness of weak solutions,
the explicit nonlinearity, and an inherent nonconvex modulus constraint.
Numerical approximation schemes for weak solutions of LLG are first proposed in~\cite{aj2006,bkp2008}.
First unconditional convergence results can be found in~\cite{bp2006,alouges2008a}, which consider the small-particle limit of LLG with exchange only.
On the one hand, the integrator of~\cite{bp2006} relies on the midpoint rule and reduced integration, and thus has to solve one nonlinear system of equations per time-step.
On the other hand, the tangent plane integrator of~\cite{alouges2008a}, which extends the prior works~\cite{aj2006,bkp2008}, relies on a reformulation of LLG which is solved for the discrete time derivative.
Each time-step consists of the solution of one linear system of equations plus nodal projection.
It has been generalized to linear-implicit time integration and full effective field in~\cite{akt2012,bffgpprs2014}.
\par Numerical integration of the coupling of LLG to other time-dependent PDEs has been analyzed in~\cite{bbp2008,bpp2013} for the full Maxwell equations (MLLG), in~\cite{lt2013,lppt2013} for the eddy current formulation, and in~\cite{bppr2013} for LLG with magnetostriction.
While~\cite{bbp2008} analyzes an extension of the midpoint scheme of~\cite{bp2006}, the works~\cite{bpp2013,lt2013,lppt2013,bppr2013} extend the tangent plane scheme from~\cite{alouges2008a}, and emphasis is on the decoupling of the time-marching scheme in~\cite{bpp2013,lppt2013,bppr2013}.
\par In the models and works mentioned, e.g., MLLG, the coupling of LLG and Maxwell equations is weak in the sense that the magnetization of LLG only contributes to the right-hand side of the Maxwell system, while the magnetic field from the Maxwell equations gives a contribution to the effective field of LLG.
In SDLLG the principal part of the differential operator of the spin diffusion equation depends nonlinearly on the magnetization.
A first numerical integrator for SDLLG is proposed and empirically validated in~\cite{gcw2007b}.
While this scheme appears to be unconditionally stable, the work does not prove convergence of the discrete solution towards a weak solution of SDLLG.
\par In our work, we extend the tangent plane integrator to SDLLG and prove unconditional convergence.
Altogether, the contributions of the current work can be summarized as follows:
\begin{itemize}
\item The proposed integrator is proven to converge (at least for a subsequence) towards a weak solution of SDLLG.
This convergence is unconditional, i.e., there is no CFL-type coupling of the time and space discretizations.
Despite the nonlinearity of SDLLG, each time-step requires only the solution of two successive linear systems, one for (the discrete time derivative of) the magnetization and one for the spin accumulation.
\item Our analysis thus provides, in particular, an alternate proof for the existence of (global) weak solutions of SDLLG, which has first been proved in~\cite{gcw2007a}.
In addition to~\cite{gcw2007a}, we prove that any weak limit of the proposed integrator satisfies an energy estimate similar to the theoretical behavior of (formal) strong solutions of SDLLG.
\item Unlike prior work on the tangent plane integrator, we adopt an idea from~\cite{bartels2013} and show that the nodal projection step of the tangent plane scheme is \emph{not} necessary.
In particular and unlike the cited works, our analysis can therefore avoid a technical angle condition on the triangulations used.
This result also transfers to the models and analysis of~\cite{alouges2008a,akt2012,bffgpprs2014,bpp2013,bppr2013,lppt2013} and simplifies their (extended) 
tangent plane integrators.
\end{itemize}
%%%%%%%%%%%%%%%%%%%%%%%%%%%%%%%%%%%%%%%%%%%%%%%%%%%%%%%%%%%%%%%%%%%%%%%%%%%%%%%%%%%%%%%%%%%%%%%%%%%%
\subsection{Outline}
%%%%%%%%%%%%%%%%%%%%%%%%%%%%%%%%%%%%%%%%%%%%%%%%%%%%%%%%%%%%%%%%%%%%%%%%%%%%%%%%%%%%%%%%%%%%%%%%%%%%
\noindent The paper is organized as follows:
In Section~\ref{sec:problem}, we introduce and accurately describe the mathematical model, see~\eqref{eq:spin_diff_llg_system} for the nondimensional formulation of SDLLG.
In Section~\ref{sec:algo}, we formulate a decoupled time-marching scheme (Algorithm~\ref{alg}) for the numerical integration of SDLLG and prove its 
well-posedness (Proposition~\ref{prop:discrete_wellposedness}).
Section~\ref{sec:convergence} contains the main result of our work (Theorem~\ref{thm:convergence}), which states unconditional convergence of the scheme towards weak solutions of SDLLG.
Following~\cite{gcw2007a}, weak solutions of SDLLG have finite energy.
In Section~\ref{sec:energy}, we prove that any weak limit obtained by the proposed numerical integrator shows the same energy behavior as formal strong solutions of SDLLG (Theorem~\ref{thm:energy}).
Numerical examples as well as the empirical validation of the proposed algorithm are postponed to a forthcoming paper~\cite{ahprs2014b}.
%%%%%%%%%%%%%%%%%%%%%%%%%%%%%%%%%%%%%%%%%%%%%%%%%%%%%%%%%%%%%%%%%%%%%%%%%%%%%%%%%%%%%%%%%%%%%%%%%%%%
\subsection{Notation}
%%%%%%%%%%%%%%%%%%%%%%%%%%%%%%%%%%%%%%%%%%%%%%%%%%%%%%%%%%%%%%%%%%%%%%%%%%%%%%%%%%%%%%%%%%%%%%%%%%%%
\noindent We use the standard notation~\cite{evans2010} for Lebesgue and Sobolev spaces and norms.
For any domain $D$, we denote the $L^2$ scalar product by $\left(r,s \right)_D=\int_D r s$ for all $r,s \in L^2(D)$.
In the case of (spaces of) vector-valued functions, we use bold letters.
For a sequence $\left\{ x_n \right\}_{n \geq 1}$ in a Banach space $X$ and $x \in X$, we write $x_n \to x$ (resp.\ $x_n \weakto x$) in $X$ if the sequence converges strongly (resp.\ weakly) to $x$ in $X$. Similarly, we write $x_n \subto x$ (resp.\ $x_n \subweakto x$) in $X$ if there exists a subsequence of $\left\{ x_n \right\}_{n \geq 1}$ which converges strongly (resp.\ weakly) to $x$ in $X$.
Throughout the paper, $C$ denotes a generic positive constant, independent of the discretization parameters, not necessarily the same at each occurrence.
Alternatively, we write $A \lesssim B$ to abbreviate $A \leq C\,B$.
Given $\aaa,\bbb \in \R ^3$, we denote by $\aaa\otimes\bbb\in\R ^{3 \times 3}$ the tensor product defined by $\left(\aaa\otimes\bbb\right)_{jk}=a_j b_k$ for all $1 \leq j,k \leq 3$.
By $\abs{\cdot}$, we denote both the Frobenius norm of a matrix and the Euclidean norm of a vector.
Since the meaning is clear from the argument, this does not lead to any ambiguity.
%%%%%%%%%%%%%%%%%%%%%%%%%%%%%%%%%%%%%%%%%%%%%%%%%%%%%%%%%%%%%%%%%%%%%%%%%%%%%%%%%%%%%%%%%%%%%%%%%%%%
\section{Model problem}
\label{sec:problem}
%%%%%%%%%%%%%%%%%%%%%%%%%%%%%%%%%%%%%%%%%%%%%%%%%%%%%%%%%%%%%%%%%%%%%%%%%%%%%%%%%%%%%%%%%%%%%%%%%%%%
\noindent In this section, we present the mathematical model, for which we introduce a nondimensional formulation, as well as the notion of a weak solution.
We use physical units in the International System of Units (SI).
%%%%%%%%%%%%%%%%%%%%%%%%%%%%%%%%%%%%%%%%%%%%%%%%%%%%%%%%%%%%%%%%%%%%%%%%%%%%%%%%%%%%%%%%%%%%%%%%%%%%
\subsection{Physical background}
%%%%%%%%%%%%%%%%%%%%%%%%%%%%%%%%%%%%%%%%%%%%%%%%%%%%%%%%%%%%%%%%%%%%%%%%%%%%%%%%%%%%%%%%%%%%%%%%%%%%
\noindent We consider a magnetic multilayer.
Let $\omega\subset\Omega$ be polyhedral Lipschitz domains in $\R ^3$, where $\Omega$ corresponds to the volume occupied by the multilayer, and $\omega$ corresponds to the ferromagnetic part.
A possible experimental setup is shown in~Figure~\ref{fig:multilayer}.
\begin{figure}[ht] \label{fig:multilayer}
\begin{center}
\includegraphics[height=18mm]{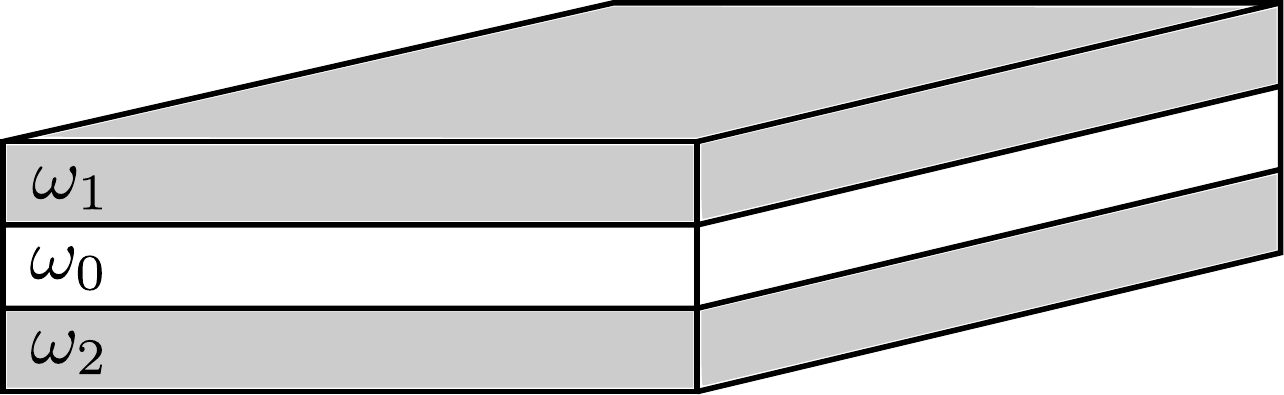}
\end{center}
\caption{Schematic of a magnetic nanopillar structure (trilayer) consisting of two ferromagnetic films, $\omega_1$ and $\omega_2$, separated by a nonmagnetic interlayer $\omega_0$. The current is assumed to flow perpendicularly from $\omega_1$ to a bottom electrode connected to $\omega_2$. In this case, $\omega=\omega_1 \cup \omega_2$ and $\Omega=\omega_1 \cup \omega_0 \cup \omega_2$.}
\end{figure}
Given some finite time $T>0$, we consider the time-space domains $\omega_T := (0,T) \times \omega$ and $\Omega_T := (0,T) \times \Omega$.
\par In micromagnetics, the quantity of interest is the magnetization $\MMM:\omega_T \rightarrow \R^3$, measured in ampere per meter (\si{\ampere\per\meter}).
If the temperature is constant and far below from the Curie temperature of the ferromagnetic material, $\MMM$ is a vector field of constant modulus $\abs{\MMM}=\Ms$, with $\Ms$ being the saturation magnetization (in \si{\ampere\per\meter}).
In the absence of spin currents, the dynamics of $\MMM$ is described by the Landau-Lifshitz-Gilbert equation (LLG), which, in the so-called Gilbert form, reads
\begin{equation} \label{eq:Llg}
\MMMt = -\gamma\mu_0 \MMM \times \HEFF + \frac{\alpha}{\Ms}\MMM \times \MMMt \quad \text{ in } \omega_T.
\end{equation}
Here, $\gamma=$ \SI{1.76e11}{\radian\per\second\per\tesla} (radian per second per tesla) and $\mu_0=$ \SI{4\pi e-7}{\newton\per\ampere\squared} (newton per square ampere) are the gyromagnetic ratio and the permeability of vacuum, respectively, while $\alpha>0$ is the nondimensional empiric Gilbert damping parameter.
The effective field $\HEFF:\Omega_T \rightarrow \R^3$, measured in \si{\ampere\per\meter}, depends on $\MMM$ and is proportional to the negative functional derivative of the total magnetic Gibbs free energy with respect to $\MMM$, i.e.,
\begin{equation}  \label{eq:func_deriv}
\mu_0 \HEFF(\MMM) = - \frac{\delta \EE (\MMM)}{\delta \MMM}.
\end{equation}
In~\eqref{eq:func_deriv} the energy functional reads
\begin{equation} \label{eq:ll_energy}
\EE (\MMM)
= \frac{A}{\Ms^2} \int_{\omega}\abs{\nabla\MMM}^2
+ K \int_{\omega}\phi\left(\MMM/\Ms\right)
- \mu_0 \int_{\omega}\HHH_e\cdot\MMM
+ \frac{\mu_0}{2} \int_{\R^3}\abs{\nabla u}^2
\end{equation}
and consists of four terms, which correspond to the exchange energy, the anisotropy energy, Zeeman's energy, and the magnetostatic energy, respectively.
In~\eqref{eq:ll_energy}, $A>0$ is the so-called exchange stiffness constant, measured in joule per meter (\si{\joule\per\meter}), and $K>0$ is the anisotropic constant (in \si{\joule\per\meter\cubed}), while $\phi:\mathbb{S}^2\to\R$ is a (nondimensional) smooth function, which takes into account the anisotropy of the ferromagnetic material.
Moreover, $\HHH_e$ is a given external field (in \si{\ampere\per\meter}), while $u:\R ^3 \to \R$ refers to the magnetostatic potential, which is the unique solution of the full-space transmission problem
\begin{equation*}% \label{eq:strayField}
\aligned
&\Delta u = \nabla\cdot\MMM  && \text{in } \omega, \\
&\Delta u = 0 && \text{in } \R^3\setminus\overline\omega, \\
&\left[ u \right]= 0 && \text{on } \partial\omega, \\
&\left[ \partial_{\nnn} u \right]= -\MMM\cdot\nnn && \text{on } \partial\omega, \\
&u(\xxx)= \bigO{1/\abs{\xxx}} && \text{as } \abs{\xxx}\to\infty.
\endaligned
\end{equation*}
Combining~\eqref{eq:func_deriv} and~\eqref{eq:ll_energy}, we obtain the following expression for the effective field
\begin{equation} \label{eq:Eff_field}
\HEFF(\MMM) = \frac{2 A}{\mu_0 \Ms^2} \Delta\MMM - \frac{K}{\mu_0 \Ms} \nabla\phi\left(\MMM/\Ms\right) + \HHH_e + \HHH_s,
\end{equation}
where $\HHH_s=-\nabla u$ denotes the stray field (in \si{\ampere\per\meter}).
\par The dynamics of the spin accumulation $\SSS:\Omega_T \rightarrow \R^3$, measured in \si{\ampere\per\meter}, is described by the diffusion equation
\begin{equation} \label{eq:Spin_diff}
\SSSt = -\nabla\cdot\JJJ_S - \frac{2 \wtd D_0}{\lambda_{sf}^2} \SSS - \frac{2 \wtd D_0}{\Ms \lambda_J^2} \SSS \times \MMM \quad \text{ in } \Omega_T,
\end{equation}
where $\wtd D_0:\Omega \to \R$ is the diffusion coefficient (in \si{\square\meter\per\second}), $\lambda_{sf}$ is the characteristic length of the spin-flip relaxation, and $\lambda_J$ is related to the mean free path of an electron (both measured in \si{\meter}).
The spin current $\JJJ_S:\Omega_T \to \R^{3 \times 3}$, measured in \si{\ampere\per\second}, is defined by
\begin{equation} \label{eq:Spin_current}
\JJJ_S = \frac{\beta\mu_B}{e \Ms} \MMM \otimes \JJJ_e - 2 \wtd D_0 \left( \nabla \SSS - \frac{\beta \beta'}{\Ms^2} \MMM \otimes \left( \nabla\SSS \cdot \MMM \right) \right) \quad \text{ in } \Omega_T,
\end{equation}
where $\mu_B=$ \SI{9.2741e-24}{\ampere\square\meter} is the Bohr magneton, $e=$ \SI{-1.602e-19}{\ampere\second} is the charge of the electron, and $\JJJ_e:\Omega_T \to \R^3$ is the applied current density field (in \si{\ampere\per\square\meter}), while the constants $0<\beta,\beta'<1$ are the nondimensional spin polarization parameters of the magnetic layers.
In~\eqref{eq:Spin_current} we denote by $\nabla\SSS \cdot \MMM \in \R^3$ the matrix-vector product between the transpose of the Jacobian $\nabla\SSS$ and $\MMM$, i.e., $\nabla\SSS \cdot \MMM= M_1 \nabla S_1 +  M_2 \nabla S_2 + M_3 \nabla S_3$.
In~\eqref{eq:Spin_diff}--\eqref{eq:Spin_current}, it is implicitly assumed that $\MMM=0$ in the nonmagnetic but conducting material $\Omega\setminus\overline\omega$.
\par To describe the dynamics of the magnetization, we take into account the interaction between the spin accumulation and the magnetization.
Thus, we consider an augmented version of~\eqref{eq:Llg}, namely
\begin{equation} \label{eq:Spin_llg}
\MMMt = -\gamma \MMM \times \left( \mu_0 \HEFF(\MMM) + J\SSS \right) + \frac{\alpha}{\Ms}\MMM \times \MMMt \quad \text{ in } \omega_T,
\end{equation}
where the constant $J$ in \si{\newton\per\square\ampere} is the strength of the interaction between the spin accumulation and the magnetization.
Finally, to complete the setting, \eqref{eq:Spin_diff}--\eqref{eq:Spin_llg} are supplemented by initial conditions
\begin{equation*}
\MMM(0) = \MMM^0  \text{ in } \omega \quad \text{ and } \quad \SSS(0) = \SSS^0  \text{ in } \Omega,
\end{equation*}
for some given initial states $\MMM^0:\omega\to\R ^3$ and $\SSS^0:\Omega\to\R ^3$ with $\abs{\MMM^0}=\Ms$, and homogeneous Neumann boundary conditions
\begin{equation*}
\frac{\partial\MMM}{\partial\nnn} = 0  \text{ on } (0,T)\times\partial\omega
\quad \text{ and } \quad \frac{\partial\SSS}{\partial\nnn} = 0  \text{ on } (0,T)\times\partial\Omega.
\end{equation*}
%%%%%%%%%%%%%%%%%%%%%%%%%%%%%%%%%%%%%%%%%%%%%%%%%%%%%%%%%%%%%%%%%%%%%%%%%%%%%%%%%%%%%%%%%%%%%%%%%%%%
\subsection{Nondimensional form of the problem}
%%%%%%%%%%%%%%%%%%%%%%%%%%%%%%%%%%%%%%%%%%%%%%%%%%%%%%%%%%%%%%%%%%%%%%%%%%%%%%%%%%%%%%%%%%%%%%%%%%%%
\noindent We introduce a nondimensional form of the system~\eqref{eq:Spin_diff}--\eqref{eq:Spin_llg}.
We perform the substitution $t'=\gamma\mu_0 \Ms t$, with $t'$ being the so-called (nondimensional) reduced time, and set $T'=\gamma\mu_0 \Ms T$.
We rescale the spatial variable by $\xxx'=\xxx/L$, with $L$ being a characteristic length of the problem (measured in \si{\meter}), e.g., the intrinsic length scale $L=\sqrt{2A/\mu_0\Ms^2}$.
However, to simplify our notation, we write $t$, $T$, $\xxx$, $\omega$, and $\Omega$, instead of $t'$, $T'$, $\xxx'$, $\omega/L$, and $\Omega/L$, respectively.
We introduce the nondimensional vector unknowns $\mmm=\MMM/\Ms$, so that the modulus constraint becomes $\abs{\mmm}=1$, and $\sss=\SSS/\Ms$.
Furthermore, we set $\heff=\HEFF/\Ms$, $\fff=\HHH_e/\Ms$, $\hhh_s=\HHH_s/\Ms$, $c=J/\mu_0$, $\jjj=\mu_B \JJJ_e / (L e \gamma \mu_0 \Ms^2)$, $D_0=2\wtd D_0 / (\gamma\mu_0 \Ms L^2)$, $\mmm^0=\MMM^0/\Ms$ and $\sss^0=\SSS^0/\Ms$.
With these substitutions, the nondimensional augmented form of LLG becomes
\begin{equation*}
\mmmt = - \mmm \times \left( \heff(\mmm) + c \sss \right) + \alpha \mmm \times \mmmt \quad \text{ in } \omega_T,
\end{equation*}
where the effective field is given by
\begin{equation} \label{eq:eff_field}
\heff(\mmm) = \Ce\Delta\mmm - \Cani\nabla\phi\left(\mmm\right) + \fff + \hhh_s(\mmm),
\end{equation}
with $\Ce=2A/(\mu_0 L^2 \Ms^2)$ and $\Cani=K/(\mu_0 \Ms^2)$, while the diffusion equation~\eqref{eq:Spin_diff} reads
\begin{equation*}
\ssst = -\nabla\cdot\left( \beta \mmm \otimes \jjj - D_0 \left( \nabla \sss - \beta \beta' \mmm \otimes \left( \nabla\sss \cdot \mmm \right) \right) \right)- \frac{L^2 D_0}{\lambda_{sf}^2} \sss - \frac{L^2 D_0}{\lambda_J^2} \sss \times \mmm \quad \text{ in } \Omega_T.
\end{equation*}
To simplify our notation and without loss of generality, we assume that $L=\lambda_{sf}=\lambda_J$.
\par To sum up, we seek for $\mmm:\omega_T \to \R ^3$ with $\abs{\mmm}=1$ and $\sss:\Omega_T \to \R ^3$ such that
\begin{subequations}\label{eq:spin_diff_llg_system}
\begin{align}
\label{eq:spin_llg}
&\mmmt = - \mmm \times \left( \heff(\mmm) + c \sss \right) + \alpha \mmm \times \mmmt && \quad\text{in } \omega_T, \\
\label{eq:spin_diff}
\begin{split}
&\ssst = -\nabla\cdot\left( \beta \mmm \otimes \jjj - D_0 \left( \nabla \sss
- \beta \beta' \mmm \otimes \left( \nabla\sss \cdot \mmm \right) \right) \right) \\
&\quad \qquad -D_0 \sss - D_0 \left(\sss \times \mmm \right)
\end{split}
&& \quad\text{in } \Omega_T, \\
&\partial_\nnn\mmm = 0  && \quad\text{on } (0,T)\times\partial\omega, \\
&\partial_\nnn\sss = 0  && \quad\text{on } (0,T)\times\partial\Omega, \\
&\mmm(0) = \mmm^0  && \quad\text{in } \omega, \\
\label{eq:bound_cond_s}
&\sss(0) = \sss^0  && \quad\text{in } \Omega.
\end{align}
\end{subequations}
Here, $c,\alpha>0$ and $0<\beta,\beta'<1$ are constants.
For the diffusion coefficient $D_0 \in L^{\infty}(\Omega)$, we assume that there exists a positive constant $D_*$ such that $D_0 \geq D_*$ a.e.\ in $\Omega$.
We also assume that $\fff \in \LLL^2(\omega_T)$ and $\jjj \in L^2(0,T;\HHH^1(\Omega))$.
Moreover, in~\eqref{eq:spin_llg} we allow a more general effective field of the form
\begin{equation} \label{eq:gen_eff_field}
\heff(\mmm) = \Ce \Delta\mmm + \ppi(\mmm) + \fff,
\end{equation}
where $\ppi:\LLL^2(\omega)\to \LLL^2(\omega)$ is a general time-independent field contribution.
We emphasize that~\eqref{eq:gen_eff_field} in particular covers~\eqref{eq:eff_field} with $\ppi(\mmm)=-\Cani\nabla\phi(\mmm) + \hhh_s(\mmm)$.
%%%%%%%%%%%%%%%%%%%%%%%%%%%%%%%%%%%%%%%%%%%%%%%%%%%%%%%%%%%%%%%%%%%%%%%%%%%%%%%%%%%%%%%%%%%%%%%%%%%%
\begin{remark}
The constraint $\abs{\mmm}=1$ directly follows from the PDE formulation, provided $\abs{\mmm^0}=1$  in $\omega_T$.
Indeed, from~\eqref{eq:spin_llg}, we deduce that $\partial_t \abs{\mmm}^2 = 2 \mmm \cdot \mmmt = 0$ in $\omega_T$.
\end{remark}
%%%%%%%%%%%%%%%%%%%%%%%%%%%%%%%%%%%%%%%%%%%%%%%%%%%%%%%%%%%%%%%%%%%%%%%%%%%%%%%%%%%%%%%%%%%%%%%%%%%%
\subsection{Weak solution of the problem}
%%%%%%%%%%%%%%%%%%%%%%%%%%%%%%%%%%%%%%%%%%%%%%%%%%%%%%%%%%%%%%%%%%%%%%%%%%%%%%%%%%%%%%%%%%%%%%%%%%%%
\noindent Let $\wtd\HHH^{-1}(\Omega)=\HHH^1(\Omega)^*$ be the dual space of $\HHH^1(\Omega)$ and denote by $\left\langle\cdot,\cdot\right\rangle$ the corresponding duality pairing, understood in the sense of the Gelfand triple $\HHH^1(\Omega) \subset \LLL^2(\Omega) \subset \wtd\HHH^{-1}(\Omega)$.
In view of the weak formulation of~\eqref{eq:spin_diff}, we consider the time-dependent bilinear form $a(t,\cdot,\cdot):\HHH^1(\Omega)\times\HHH^1(\Omega)\to\R$ defined by
\begin{equation*}
\begin{split}
a(t,\zzeta_1,\zzeta_2)=
& \left(D_0 \nabla\zzeta_1,\nabla\zzeta_2\right)_{\Omega}
- \beta\beta'\left( D_0 \mmm(t)\otimes\left(\nabla\zzeta_1\cdot\mmm(t)\right),\nabla\zzeta_2\right)_{\omega} \\
&\quad +\left( D_0 \zzeta_1,\zzeta_2\right)_{\Omega}
+\left( D_0 \left(\zzeta_1 \times \mmm(t) \right),\zzeta_2\right)_{\omega},
\end{split}
\end{equation*}
for all $t \in [0,T]$ and $\zzeta_1,\zzeta_2 \in \HHH^1(\Omega)$.
\par We recall from~\cite[Definition 1]{gcw2007a} the notion of a weak solution of the SDLLG system~\eqref{eq:spin_diff_llg_system}, which extends the definition of weak solutions of LLG from~\cite{as1992}.
%%%%%%%%%%%%%%%%%%%%%%%%%%%%%%%%%%%%%%%%%%%%%%%%%%%%%%%%%%%%%%%%%%%%%%%%%%%%%%%%%%%%%%%%%%%%%%%%%%%%
\begin{definition} \label{def:weak_sol}
Let $\mmm^0 \in \HHH^1(\omega)$ with $\abs{\mmm^0}=1$ a.e.\ in $\omega$, and $\sss^0 \in \HHH^1(\Omega)$.
The tupel $(\mmm,\sss)$ is called a weak solution of SDLLG if the following properties \rm{(i)--(iv)} are satisfied:
\begin{itemize}
\item[(i)] $\mmm \in \HHH^1(\omega_T)$ with $\abs{\mmm}=1$ a.e.\ in $\omega_T$ and $\mmm(0)=\mmm^0$ in the sense of traces,
\item[(ii)] $\sss \in L^2(0,T;\HHH^1(\Omega)) \cap L^{\infty}(0,T;\LLL^2(\Omega))$, $\ssst \in L^2(0,T;\wtd\HHH^{-1}(\Omega))$ and $\sss(0)=\sss^0$ in the sense of traces,
\item[(iii)] for all $\vphi \in \HHH^1(\omega_T)$, it holds
\begin{subequations}
\begin{equation} \label{eq:weak_spin_llg}
\begin{split}
&\left(\mmmt,\vphi\right)_{\omega_T}
+ \alpha\left(\mmmt \times \mmm,\vphi\right)_{\omega_T} \\
&\quad = -\Ce \left(\nabla\mmm\times\mmm,\nabla\vphi \right)_{\omega_T}
+ \left(\ppi(\mmm)\times \mmm,\vphi \right)_{\omega_T}
+ \left(\fff\times\mmm,\vphi\right)_{\omega_T}
+ c \left(\sss\times\mmm,\vphi\right)_{\omega_T},
\end{split}
\end{equation}
\item[(iv)] for almost all $t \in [0,T]$ and all $\zzeta \in \HHH^1(\Omega)$, it holds
\begin{equation} \label{eq:weak_spin_diff}
\left\langle \ssst(t),\zzeta \right\rangle
+ a(t,\sss(t),\zzeta)
= \beta \left(\mmm(t)\otimes\jjj(t),\nabla\zzeta\right)_{\omega}
- \beta \left(\jjj(t)\cdot\nnn,\mmm(t)\cdot\zzeta\right)_{\partial\Omega\cap\partial\omega}.
\end{equation}
\end{subequations}
\end{itemize}
\end{definition}
%%%%%%%%%%%%%%%%%%%%%%%%%%%%%%%%%%%%%%%%%%%%%%%%%%%%%%%%%%%%%%%%%%%%%%%%%%%%%%%%%%%%%%%%%%%%%%%%%%%%
\begin{remark}
If $(\mmm,\sss)$ is a weak solution of SDLLG, then it holds $\mmm\in C(0,T;\LLL^2(\omega))$ and $\sss\in C(0,T;\LLL^2(\Omega))$, cf., e.g., \cite[Section~5.9.2, Theorem~2 and Theorem~3]{evans2010}.
\end{remark}
%%%%%%%%%%%%%%%%%%%%%%%%%%%%%%%%%%%%%%%%%%%%%%%%%%%%%%%%%%%%%%%%%%%%%%%%%%%%%%%%%%%%%%%%%%%%%%%%%%%%
\begin{remark}
The boundary term in~\eqref{eq:weak_spin_diff} is missing in~\cite{gcw2007a}. This error has recently been noticed and corrected, so that the overall result of~\cite{gcw2007a} remains valid~\cite{garciacervera2014}.
The present analysis provides an alternate proof for the existence of solutions and hence validity of the results of~\cite{gcw2007a,garciacervera2014}.
\end{remark}
\noindent The following lemma highlights the parabolic nature of equation~\eqref{eq:spin_diff}.
%%%%%%%%%%%%%%%%%%%%%%%%%%%%%%%%%%%%%%%%%%%%%%%%%%%%%%%%%%%%%%%%%%%%%%%%%%%%%%%%%%%%%%%%%%%%%%%%%%%%
\begin{lemma} \label{lem:coercivity}
The time-dependent bilinear form $a(t,\cdot,\cdot)$ is continuous and positive definite.
Indeed, it holds
\begin{equation} \label{eq:coercivity_a}
a(t,\zzeta,\zzeta)
\geq (1 - \beta\beta') D_* \norm{\zzeta}{\HHH^1(\Omega)}^2
\end{equation}
for almost all $t \in [0,T]$.
\end{lemma}
%%%%%%%%%%%%%%%%%%%%%%%%%%%%%%%%%%%%%%%%%%%%%%%%%%%%%%%%%%%%%%%%%%%%%%%%%%%%%%%%%%%%%%%%%%%%%%%%%%%%
\begin{proof}
The continuity directly follows from the regularity assumptions on the data, as $\abs{\mmm}=1$ a.e.\ in $\Omega_T$.
As for the positive definiteness, we note
\begin{equation*}
\abs{\mmm(t)\otimes\left(\nabla\zzeta\cdot\mmm(t)\right)\cdot\nabla\zzeta}
\leq \abs{\nabla\zzeta}^2
\quad \text{for all } \zzeta\in\HHH^1(\Omega).
\end{equation*}
As a consequence, since $D_0 \geq D_*$ and $0 < \beta\beta' <1$, we get
\begin{equation*}
\begin{split}
a(t,\zzeta,\zzeta)
& = \left(D_0 \nabla\zzeta,\nabla\zzeta\right)_{\Omega}
- \beta\beta'\left( D_0 \mmm(t)\otimes\left(\nabla\zzeta\cdot\mmm(t)\right),\nabla\zzeta\right)_{\omega}
+\left( D_0 \zzeta,\zzeta\right)_{\Omega} \\
& \geq \left(D_0 \nabla\zzeta,\nabla\zzeta\right)_{\Omega}
- \beta\beta'\left(D_0\abs{\mmm(t)\otimes\left(\nabla\zzeta\cdot\mmm(t)\right)},\abs{\nabla\zzeta}\right)_{\omega}
+\left( D_0 \zzeta,\zzeta\right)_{\Omega} \\
& \geq (1 - \beta\beta') D_* \norm{\nabla\zzeta}{\LLL^2(\Omega)}^2
+ D_* \norm{\zzeta}{\LLL^2(\Omega)}^2.
\end{split}
\end{equation*}
This establishes~\eqref{eq:coercivity_a} and concludes the proof.
\end{proof}
%%%%%%%%%%%%%%%%%%%%%%%%%%%%%%%%%%%%%%%%%%%%%%%%%%%%%%%%%%%%%%%%%%%%%%%%%%%%%%%%%%%%%%%%%%%%%%%%%%%%
\section{Numerical algorithm}
\label{sec:algo}
%%%%%%%%%%%%%%%%%%%%%%%%%%%%%%%%%%%%%%%%%%%%%%%%%%%%%%%%%%%%%%%%%%%%%%%%%%%%%%%%%%%%%%%%%%%%%%%%%%%%
\noindent For the time discretization, we consider a uniform partition $0 = t_0 < t_1 < \dots < t_N = T$ of the time interval $[0,T]$ with time-step size $k=T/N$, i.e., $t_j=jk$ for $0 \leq j \leq N$.
\par Given a sequence of functions $\left\{\vphi^j\right\}_{0 \leq j \leq N}$, such that any $\vphi^j$ is associated with the time-step $t_j$, for $0 \leq j \leq N-1$ we define the difference quotient $\dt\vphi^{j+1} := (\vphi^{j+1} - \vphi^j)/k$.
We consider the piecewise linear and the two piecewise constant time-approximations defined as follows:
For $0 \leq j \leq N-1$ and $t \in [t_j,t_{j+1})$, we have
\begin{equation} \label{eq:time_approx}
\vphi_{k}(t) := \frac{t-t_j}{k}\vphi^{j+1} + \frac{t_{j+1} - t}{k}\vphi^j, \quad \vphi_{k}^-(t):= \vphi^j, \quad \vphi^+(t):= \vphi^{j+1}.
\end{equation}
Obviously, it holds $\partial_t \vphi_k(t) = \dt \vphi^{j+1}$ for all $t \in [t_j,t_{j+1})$.
\par For the spatial discretization, let $\left\{\TT_h^{\Omega}\right\}_{h>0}$ be a shape-regular and (globally) quasi-uniform family of regular tetrahedral triangulations of $\Omega$, parameterized by the meshsize $h=\max h_K$, where $h_K = \diam(K)$ for all $K \in \TT_h^{\Omega}$.
By $\TT_h^{\omega}$, we denote the restriction of $\TT_h^{\Omega}$ to $\omega$.
We assume that $\omega$ is resolved, i.e.,
\begin{equation*}
\TT_h^{\omega} = {\TT_h^{\Omega}}\vert_{\omega} = \set{K \in \TT_h^{\Omega}}{K \cap \omega \neq \emptyset} \quad \text{and} \quad \overline \omega = \bigcup_{K \in \TT_h^{\omega}} K.
\end{equation*}
Let us denote by $\SS^1(\TT_h^{\Omega})^3$ the standard finite element space of globally continuous and piecewise affine functions from $\Omega$ to $\R^3$.
%, namely
% \begin{equation*}
% \SS^1(\TT_h^{\Omega})^3 := \left\{\pphi_h \in \CCC(\overline\Omega) : {\pphi_h}\vert_K \text{ is affine for all } K \in \TT_h^{\Omega}\right\}.
% \end{equation*}
Correspondingly, we also consider $\SS^1(\TT_h^{\omega})^3$.
By $\II_h^{\Omega}: \CCC(\overline\Omega) \to \SS^1(\TT_h^{\Omega})^3$ and $\II_h^{\omega}: \CCC(\overline\omega) \to \SS^1(\TT_h^{\omega})^3$, we denote the nodal interpolation operators onto these spaces.
Since $\omega$ is resolved, these operators coincide on $\omega$, i.e., $\II_h^{\Omega}(\vphi)\vert_{\omega}=\II_h^{\omega}\left(\vphi\vert_{\omega}\right)$ for all $\vphi \in \CCC(\overline\Omega)$.
In particular, there is no ambiguity, if we denote both operators by $\II_h$.
The set of nodes of the triangulation $\TT_h^{\omega}$ is denoted by $\NN_h^{\omega}$.
\par We recall that, under the constraint $\abs{\mmm}=1$, the strong form of~\eqref{eq:spin_llg} can equivalently be stated as
\begin{equation} \label{eq:spin_llg_alg}
\alpha \mmmt + \mmm \times \mmmt = \heff(\mmm) + c \sss - \left(\left(\heff(\mmm) + c \sss \right)  \cdot \mmm \right)\mmm.
\end{equation}
This formulation is used to construct the upcoming numerical scheme.
Since~\eqref{eq:spin_llg_alg} is linear in $\mmmt$, the main idea is to introduce an additional free variable $\vvv = \mmmt$.
To discretize $\vvv$, we introduce the discrete tangent space defined by
\begin{equation*}
\KK_{\pphi_h} := \left\{\ppsi_h \in \SS^1(\TT_h^{\omega})^3 : \ppsi_h(\zzz) \cdot \pphi_h(\zzz) = 0 \text{ for all } \zzz \in \NN_h^{\omega}\right\}
\end{equation*}
for any $\pphi_h \in \SS^1(\TT_h^{\omega})^3$.
Moreover, we consider the set
\begin{equation*}
\MM_h := \left\{\pphi_h \in \SS^1(\TT_h^{\omega})^3: \abs{\pphi_h(\zzz)}=1 \text{ for all } \zzz \in \NN_h^{\omega}\right\}.
\end{equation*}
These sets reflect two main properties of $\mmm$ and $\vvv$, namely the orthogonality $\mmm\cdot\vvv=0$ and the unit-length constraint $\abs{\mmm}=1$.
\par Let $\UU_h=\left\{\pphi_h \in \SS^1(\TT_h^{\omega})^3: \abs{\pphi_h(\zzz)} \geq 1 \text{ for all } \zzz \in \NN_h^{\omega}\right\}$.
We consider the nodal projection map $\Pi_h:\UU_h\to\MM_h$ defined by $\Pi_h\pphi_h(\zzz) = \pphi_h(\zzz)/\abs{\pphi_h(\zzz)}$ for all $\zzz \in \NN_h^{\omega}$ and $\pphi_h\in\UU_h$.
A simple argument based on the elementwise use of barycentric coordinates shows that $\norm{\Pi_h\pphi_h}{\LLL^\infty(\omega)} = 1$ for all $\pphi_h\in\UU_h$.
Moreover, we have the estimate
\begin{equation} \label{eq:energyDecay}
\norm{\nabla\Pi_h\pphi_h}{\LLL^2(\omega)} \leq c_{\Pi} \norm{\nabla\pphi_h}{\LLL^2(\omega)} \quad \text{for all } \pphi_h \in \UU_h,
\end{equation}
where the constant $c_{\Pi}>0$ depends only on the shape-regularity of the triangulation, cf., e.g.,~\cite[Lemma 2.2]{bartels2013}.
With an additional angle condition on $\TT_h^{\omega}$, it is well known that~\eqref{eq:energyDecay} holds even with $c_{\Pi}=1$, cf.~\cite{bartels2005}.
\par Let $\mmmh^0 \in \MM_h$ and $\sssh^0 \in \SS^1(\TT_h^{\Omega})^3$ be suitable approximations of the initial conditions.
Moreover, we consider a numerical realization $\ppih:\LLL^2(\omega) \to \LLL^2(\omega)$ of $\ppi$, which is assumed to fulfill a certain set of properties, see~(H2)--(H3) below.
This allows us to include the approximation errors, e.g., those which arise from the numerical computation of the stray field, into the overall convergence analysis.
For ease of presentation, we assume that $\fff$ and $\jjj$ are continuous in time, i.e., $\fff \in C(0,T;\LLL^2(\omega))$ and $\jjj \in C(0,T;\HHH^1(\Omega))$, so that the expressions $\fff^j=\fff(t_j)$ and $\jjj^j=\jjj(t_j)$ are meaningful for all $0 \leq j \leq N$.
It is even possible to replace $\fff^j$ and $\jjj^j$ by some numerical approximation $\fff^j_h$ and $\jjj_h^j$ as long as some weak convergence properties are fulfilled, cf.~\cite{bffgpprs2014}.
\par Analogously to what we have done in Section~\ref{sec:problem} for the continuous problem, for $0 \leq i \leq N-1$ we define the bilinear form $a_h^{i+1}:\SS^1(\TT_h^{\Omega})^3 \times \SS^1(\TT_h^{\Omega})^3 \to \R$ by
\begin{equation*}
\begin{split}
a_h^{i+1}(\zzeta_1,\zzeta_2)
&= \left( D_0 \nabla \zzeta_1, \nabla \zzeta_2\right)_{\Omega}
- \beta\beta' \left( D_0 \hmmmh^{i+1} \otimes \left(\nabla \zzeta_1\cdot \hmmmh^{i+1} \right), \nabla \zzeta_2 \right)_{\omega} \\
&\quad + \left( D_0 \zzeta_1, \zzeta_2 \right)_{\Omega}
+ \left( D_0\left(\zzeta_1 \times \hmmmh^{i+1}\right), \zzeta_2 \right)_{\omega}
\end{split}
\end{equation*}
for all $\zzeta_1,\zzeta_2 \in \SS^1(\TT_h^{\Omega})^3$.
For the numerical integration of the SDLLG system~\eqref{eq:spin_diff_llg_system}, we propose the following algorithm.
%%%%%%%%%%%%%%%%%%%%%%%%%%%%%%%%%%%%%%%%%%%%%%%%%%%%%%%%%%%%%%%%%%%%%%%%%%%%%%%%%%%%%%%%%%%%%%%%%%%%
\begin{algorithm}\label{alg}
Input: $\mmmh^0$, $\sssh^0$, $\left\{ \fff^i \right\}_{0 \leq i \leq N-1}$, $\left\{ \jjj^{i+1} \right\}_{0 \leq i \leq N-1}$, parameter $0 \leq \theta \leq 1$.\\
For all $0 \leq i \leq N-1$ iterate:
\begin{itemize}
\item[(i)] compute $\vvvh^i \in \KK_{\mmmh^i}$ such that
\begin{subequations}
\begin{equation}\label{eq:alg:1}
\begin{split}
&\alpha\left( \vvvh^i,\pphi_h \right)_{\omega}
+\left( \mmmh^i \times \vvvh^i,\pphi_h \right)_{\omega}
+ \Ce \theta k \left( \nabla\vvvh^i, \nabla \pphi_h \right)_{\omega} \\
&\quad = -\Ce \left(\nabla\mmmh^i, \nabla\pphi_h \right)_{\omega}
+ \left( \ppih(\mmmh^i), \pphi_h \right)_{\omega}
+ \left( \fff^i, \pphi_h \right)_{\omega}
+ c \left( \sssh^i, \pphi_h \right)_{\omega}
\end{split}
\end{equation}
for all $\pphi_h \in \KK_{\mmmh^i}$;
\item[(ii)] define $\mmmh^{i+1} \in \SS^1(\TT_h)^3$ by
\begin{equation}\label{eq:alg:2}
\mmmh^{i+1} = \mmmh^i + k\vvvh^i;
\end{equation}
\item[(iii)] compute $\sssh^{i+1} \in \SS^1(\TT_h^{\Omega})^3$ such that
\begin{equation}\label{eq:alg:3}
\left( \dt \sssh^{i+1},\zzeta_h \right)_{\Omega}
+ a_h^{i+1}(\sssh^{i+1},\zzeta_h)
= \beta \left(\hmmmh^{i+1} \otimes \jjj^{i+1},\nabla\zzeta_h \right)_{\omega}
- \beta \left(\jjj^{i+1}\cdot\nnn,\hmmmh^{i+1}\cdot\zzeta_h\right)_{\partial\Omega\cap\partial\omega}
\end{equation}
\end{subequations}
for all $\zzeta_h \in \SS^1(\TT_h^{\Omega})^3$.
\end{itemize}
Output: Sequence of discrete functions $\left\{\left(\vvvh^i,\mmmh^{i+1},\sssh^{i+1}\right)\right\}_{0 \leq i \leq N-1}$.
\end{algorithm}
%%%%%%%%%%%%%%%%%%%%%%%%%%%%%%%%%%%%%%%%%%%%%%%%%%%%%%%%%%%%%%%%%%%%%%%%%%%%%%%%%%%%%%%%%%%%%%%%%%%%
\noindent The overall system~\eqref{eq:spin_diff_llg_system} is a nonlinearly coupled system of a linear diffusion equation for $\sss$ with the nonlinear LLG equation for $\mmm$.
However, our scheme only requires the solution of two linear systems per time-step, since the treatment of the micromagnetic part and the spin diffusion part is completely decoupled for the time-integration.
This greatly simplifies an actual numerical implementation as well as the possible preconditioning of iterative solvers.
%%%%%%%%%%%%%%%%%%%%%%%%%%%%%%%%%%%%%%%%%%%%%%%%%%%%%%%%%%%%%%%%%%%%%%%%%%%%%%%%%%%%%%%%%%%%%%%%%%%%
\begin{remark}
Unlike this work, earlier results on the tangent plane integrator~\cite{alouges2008a,akt2012,bffgpprs2014,bbp2008,bppr2013,lppt2013} define $\mmmh^{i+1} := \Pi_h(\mmmh^i+k\vvvh^i)$ in~\eqref{eq:alg:2}.
Unconditional convergence in the sense of Theorem~\ref{thm:convergence} can then be achieved with an additional angle condition on the triangulation $\TT_h^{\omega}$, which ensures~\eqref{eq:energyDecay} with $c_{\Pi}=1$.
This assumption is avoided in the present work.
\end{remark}
%%%%%%%%%%%%%%%%%%%%%%%%%%%%%%%%%%%%%%%%%%%%%%%%%%%%%%%%%%%%%%%%%%%%%%%%%%%%%%%%%%%%%%%%%%%%%%%%%%%%
\noindent The following result follows from standard scaling arguments.
%%%%%%%%%%%%%%%%%%%%%%%%%%%%%%%%%%%%%%%%%%%%%%%%%%%%%%%%%%%%%%%%%%%%%%%%%%%%%%%%%%%%%%%%%%%%%%%%%%%%
\begin{lemma}\label{lem:norm_equiv}
Let $\left\{\TT_h\right\}_{h>0}$ be a quasi-uniform family of triangulations of $\Omega$ and $r \in [1,\infty)$.
Then,
\begin{equation*}
C^{-1} \norm{w_h}{L^r(\Omega)}^r \leq h^3 \sum_{z \in \NN_h} \abs{w_h(z)}^r \leq C \norm{w_h}{L^r(\Omega)}^r
\quad \text{ for all } w_h \in \SS^1(\TT_h).
\end{equation*}
The constant $C>0$ depends only on $r$, but is independent of the meshsize $h$. \qed
\end{lemma}
%%%%%%%%%%%%%%%%%%%%%%%%%%%%%%%%%%%%%%%%%%%%%%%%%%%%%%%%%%%%%%%%%%%%%%%%%%%%%%%%%%%%%%%%%%%%%%%%%%%%
\noindent The following proposition states that the above algorithm is well defined, cf.~\cite[Proposition~3.1 and Proposition~4.1]{bartels2013} for corresponding results in the frame of harmonic maps and the harmonic map heat flow.
%%%%%%%%%%%%%%%%%%%%%%%%%%%%%%%%%%%%%%%%%%%%%%%%%%%%%%%%%%%%%%%%%%%%%%%%%%%%%%%%%%%%%%%%%%%%%%%%%%%%
\begin{proposition} \label{prop:discrete_wellposedness}
Algorithm~\ref{alg} is well defined in the following sense: For each time-step $0 \leq i \leq N-1$, there exists a unique solution $(\vvvh^i, \mmmh^{i+1}, \sssh^{i+1})$.
Moreover, it holds
\begin{equation} \label{eq:nodewise}
\abs{\mmmh^{i+1}(\zzz)}^2
= 1 + k^2 \sum_{\ell=0}^i \abs{\vvv_h^{\ell}(\zzz)}^2 \geq 1 \quad \text{for all } \zzz \in \NN_h^{\omega},
\end{equation}
as well as
\begin{equation} \label{eq:boundedness_mh}
C_*^{-1}\norm{\mmm_h^{i+1}}{\LLL^2(\omega)}^2
\leq 1 + k^2 \sum_{\ell=0}^i \norm{\vvvh^{\ell}}{\LLL^2(\omega)}^2,
\end{equation}
where the constant $C_*>0$ depends only on the shape-regularity of $\left\{\TT_h^{\omega}\right\}_{h>0}$, but is independent of $h$ and $k$.
\end{proposition}
%%%%%%%%%%%%%%%%%%%%%%%%%%%%%%%%%%%%%%%%%%%%%%%%%%%%%%%%%%%%%%%%%%%%%%%%%%%%%%%%%%%%%%%%%%%%%%%%%%%%
\begin{proof}
Let $0 \leq i \leq N-1$.
For step~(i) of the algorithm, it is straightforward to show that problem~\eqref{eq:alg:1} is characterized by a positive definite bilinear form.
Unique solvability thus follows from linearity and finite space dimension.
Step~(ii) is clearly well defined.
For all $\zzz \in \NN_h$, the nodewise orthogonality from $\KK_{\mmmh^i}$ proves
\begin{equation*}
\abs{\mmmh^{i+1}(\zzz)}^2
= \abs{\mmmh^i(\zzz) + k \vvv_h^i(\zzz)}^2
= \abs{\mmmh^i(\zzz)}^2 + k^2 \abs{\vvv_h^i(\zzz)}^2.
\end{equation*}
Since $\mmmh^0 \in \MM_h$, mathematical induction proves
\begin{equation*}
\abs{\mmmh^{i+1}(\zzz)}^2
= \abs{\mmmh^0(\zzz)}^2 + k^2 \sum_{\ell=0}^i \abs{\vvv_h^{\ell}(\zzz)}^2
= 1 + k^2 \sum_{\ell=0}^i \abs{\vvv_h^{\ell}(\zzz)}^2  \geq 1.
\end{equation*}
This proves~\eqref{eq:nodewise}.
The norm equivalence from Lemma~\ref{lem:norm_equiv} in the case $r=2$ yields
\begin{equation*}
\begin{split}
\norm{\mmm_h^{i+1}}{\LLL^2(\omega)}^2
& \lesssim h^3 \sum_{\zzz \in \NN_h^{\omega}} \abs{\mmm_h^{i+1}(\zzz)}^2
= h^3 \sum_{\zzz \in \NN_h^{\omega}} \left(1+k^2\sum_{\ell=0}^i\abs{\vvv_h^{\ell}(\zzz)}^2 \right) \\
& = h^3 (\#\NN_h^{\omega}) + k^2 \sum_{\ell=0}^i h^3 \sum_{\zzz \in \NN_h^{\omega}}\abs{\vvv_h^{\ell}(\zzz)}^2
\lesssim 1 + k^2 \sum_{\ell=0}^i \norm{\vvvh^{\ell}}{\LLL^2(\omega)}^2.
\end{split}
\end{equation*}
This establishes~\eqref{eq:boundedness_mh}.
For step~(iii), we use the same argument as for step~(i).
Due to~\eqref{eq:nodewise}, the nodewise projections in~\eqref{eq:alg:3} are well defined.
Let $b_h^i:\SS^1(\TT_h^{\Omega})^3 \times \SS^1(\TT_h^{\Omega})^3 \to \R$ be the bilinear form associated to problem~\eqref{eq:alg:3}, i.e.,
\begin{equation*}
\begin{split}
b_h^i(\zzeta_1,\zzeta_2)
&=\frac{1}{k} \left( \zzeta_1, \zzeta_2 \right)_{\Omega}
+ \left(D_0 \nabla \zzeta_1, \nabla \zzeta_2  \right)_{\Omega}
- \beta\beta'\left( D_0 \hmmmh^{i+1} \otimes \left(\nabla \zzeta_1\cdot \hmmmh^{i+1} \right), \nabla \zzeta_2  \right)_{\Omega} \\
&\quad + \left(D_0 \zzeta_1 , \zzeta_2  \right)_{\Omega}
+ \left(D_0\left(\zzeta_1 \times \hmmmh^{i+1}\right), \zzeta_2  \right)_{\Omega}.
\end{split}
\end{equation*}
Since $\norm{\hmmmh^{i+1}}{\LLL^\infty(\omega)}=1$, we see
\begin{equation*}
\left(D_0 \hmmmh^{i+1} \otimes \left(\nabla \zzeta_1\cdot \hmmmh^{i+1} \right) , \nabla \zzeta_2  \right)_{\Omega}
\leq \left(D_0 \abs{\nabla\zzeta_1}, \abs{\nabla\zzeta_2}\right)_{\Omega}.
\end{equation*}
It follows that
\begin{equation*}
\begin{split}
b_h^i(\zzeta,\zzeta)
&\geq\frac{1}{k} \left\Vert\zzeta\right\Vert_{\LLL^2(\Omega)}^2
+ (1 - \beta\beta') \left(D_0 \nabla\zzeta,\nabla\zzeta \right)_{\Omega}
+ \left(D_0 \zzeta,\zzeta \right)_{\Omega} \\
&\geq \frac{1+kD_*}{k} \norm{\zzeta}{\LLL^2(\Omega)}^2
+ D_* (1 - \beta\beta') \norm{\nabla\zzeta}{\LLL^2(\Omega)}^2.
\end{split}
\end{equation*}
As $0<\beta\beta'<1$ and $D_*>0$, $b_h^i(\cdot,\cdot)$ is positive definite and problem~\eqref{eq:alg:3} is thus well posed.
\end{proof}
%%%%%%%%%%%%%%%%%%%%%%%%%%%%%%%%%%%%%%%%%%%%%%%%%%%%%%%%%%%%%%%%%%%%%%%%%%%%%%%%%%%%%%%%%%%%%%%%%%%%
\section{Convergence analysis}
\label{sec:convergence}
%%%%%%%%%%%%%%%%%%%%%%%%%%%%%%%%%%%%%%%%%%%%%%%%%%%%%%%%%%%%%%%%%%%%%%%%%%%%%%%%%%%%%%%%%%%%%%%%%%%%
\noindent In this section, we consider the convergence properties of Algorithm~\ref{alg} and show that it is indeed unconditionally convergent towards a weak solution of SDLLG in the sense of Definition~\ref{def:weak_sol}.
We emphasize that the proof is constructive in the sense that it even shows existence of weak solutions.
We start by collecting some general assumptions:
\begin{itemize}
\item[(H1)] The discrete initial data $\mmm^0 \in \MM_h$ and $\sss^0 \in \SS^1(\TT_h^{\omega})^3$ satisfy
\begin{equation*}
\mmmh^0 \weakto \mmm^0 \text{ in } \HHH^1(\omega) \quad \text{ and } \quad
\sssh^0 \weakto \sss^0 \text{ in } \LLL^2(\Omega).
\end{equation*}
\item[(H2)] The general field contribution $\ppih$ is bounded, i.e.,
\begin{equation*}
\norm{\ppih(\www)}{\LLL^2(\omega)} \leq C_{\ppi}\norm{\www}{\LLL^2(\omega)} \quad \text{ for all } \www \in \LLL^2(\omega),
\end{equation*}
with a constant $C_{\ppi} > 0$ which depends only on $\abs{\omega}$.
\item[(H3)] It holds
\begin{equation*}
\ppih(\www_{hk}) \weakto \ppi(\www) \text{ in } \LLL^2(\omega_T) \quad \text{ as } (h,k) \to 0
\end{equation*}
for any sequence $\www_{hk} \to \www$ in $\LLL^2(\omega_T)$.
\end{itemize}
%%%%%%%%%%%%%%%%%%%%%%%%%%%%%%%%%%%%%%%%%%%%%%%%%%%%%%%%%%%%%%%%%%%%%%%%%%%%%%%%%%%%%%%%%%%%%%%%%%%%
\begin{remark} \label{rem:discrete_pi}
Usual stray field discretizations by hybrid FEM-BEM methods, e.g., the Fredkin-Koehler approach from~\cite{fk1990}, or FEM-BEM coupling methods satisfy~{\rm{(H2)--(H3)}}, see~\cite{bffgpprs2014}.
\end{remark}
%%%%%%%%%%%%%%%%%%%%%%%%%%%%%%%%%%%%%%%%%%%%%%%%%%%%%%%%%%%%%%%%%%%%%%%%%%%%%%%%%%%%%%%%%%%%%%%%%%%%
\begin{remark}
For a discrete operator $\ppih:\HHH^1(\omega) \to \LLL^2(\omega)$, assumption~{\rm{(H2)}} can be relaxed to
\begin{equation*}
\norm{\ppih(\www)}{\LLL^2(\omega)} \leq C_{\ppi}\norm{\www}{\HHH^1(\omega)} \quad \text{ for all } \www \in \HHH^1(\omega).
\end{equation*}
Within this setting, and with an appropriate modification of assumption~{\rm{(H3)}}, the hybrid FEM-BEM method from~\cite{gcr2006} for the computation of the stray field can also be included into our analysis.
Then, the proof of Proposition~\ref{lem:discrete_stability2} below becomes more technical, but the assertion remains true. We refer to the argument of~\cite{bffgpprs2014} which can be adapted accordingly.
\end{remark}
%%%%%%%%%%%%%%%%%%%%%%%%%%%%%%%%%%%%%%%%%%%%%%%%%%%%%%%%%%%%%%%%%%%%%%%%%%%%%%%%%%%%%%%%%%%%%%%%%%%%
\noindent From now on, we consider the time-approximations $\mmmhk$, $\mmmhk^{\pm}$, $\ssshk$, $\ssshk^{\pm}$ defined by~\eqref{eq:time_approx}.
The next theorem is the main result of this work.
%%%%%%%%%%%%%%%%%%%%%%%%%%%%%%%%%%%%%%%%%%%%%%%%%%%%%%%%%%%%%%%%%%%%%%%%%%%%%%%%%%%%%%%%%%%%%%%%%%%%
\begin{theorem}\label{thm:convergence}
Let $\left\{\TT_h^{\Omega}\right\}_{h>0}$ be a shape-regular and quasi-uniform family of triangulations.
\begin{itemize}
\item[(a)] Suppose $1/2 < \theta \leq 1$ and that assumptions~{\rm{(H1)--(H2)}} are satisfied.\\
Then, there exist $\mmm\in\LLL^2(\omega_T)$ and $\sss\in L^2(0,T;\HHH^1(\Omega))$ such that
\begin{equation*}
\mmmhk^- \subto \mmm \text{ in } \LLL^2(\omega_T) \quad \text{ and } \quad
\ssshk^- \subweakto \sss \text{ in } L^2(0,T;\HHH^1(\Omega)).
\end{equation*}
\item[(b)] In addition to the above, let assumption~{\rm{(H3)}} be satisfied.
Then, it holds
\begin{equation*}
(\mmmhk,\ssshk) \subweakto (\mmm,\sss) \text{ in } \HHH^1(\omega_T) \times \left[L^2(0,T;\HHH^1(\Omega)) \cap H^1(0,T;\wtd\HHH^{-1}(\Omega))\right],
\end{equation*}
where $(\mmm,\sss)$ is a weak solution of SDLLG.
\end{itemize}
\end{theorem}
%%%%%%%%%%%%%%%%%%%%%%%%%%%%%%%%%%%%%%%%%%%%%%%%%%%%%%%%%%%%%%%%%%%%%%%%%%%%%%%%%%%%%%%%%%%%%%%%%%%%
\begin{remark}
In particular, Theorem \ref{thm:convergence} yields existence of weak solutions, and each accumulation point of $(\mmmhk, \ssshk)$ is a weak solution of SDLLG in the sense of Definition~\ref{def:weak_sol}.
\end{remark}
%%%%%%%%%%%%%%%%%%%%%%%%%%%%%%%%%%%%%%%%%%%%%%%%%%%%%%%%%%%%%%%%%%%%%%%%%%%%%%%%%%%%%%%%%%%%%%%%%%%%
\noindent The proof of Theorem~\ref{thm:convergence} will roughly be done in three steps, namely
\begin{enumerate}
\item[(i)] boundedness of the discrete quantities and energies,
\item[(ii)] existence of weakly convergent subsequences via compactness,
\item[(iii)] identification of the limits with weak solutions of SDLLG.
\end{enumerate}
For the sake of readability, we split our argument into several lemmata.
\par To start with, we recall the following result, which states a well-known and simple algebraic trick which often simplifies the computation and the estimation of sums.
%%%%%%%%%%%%%%%%%%%%%%%%%%%%%%%%%%%%%%%%%%%%%%%%%%%%%%%%%%%%%%%%%%%%%%%%%%%%%%%%%%%%%%%%%%%%%%%%%%%%
\begin{lemma}[Abel's summation by parts] \label{lem:abel}
Let $X$ be a vector space endowed with a symmetric bilinear form $(\cdot,\cdot)$.
Given an integer $j \geq 1$, let $\left\{v_i\right\}_{0\leq i \leq j} \subset X$.
Then, it holds
\begin{equation*}
\sum_{i=0}^{j-1} (v_{i+1}-v_i,v_{i+1})
= \frac{1}{2} \left(v_j,v_j \right)
- \frac{1}{2} \left(v_0,v_0 \right)
+ \frac{1}{2} \sum_{i=0}^{j-1} \left(v_{i+1}-v_i,v_{i+1}-v_i \right). \tag*{\qed}
\end{equation*}
\end{lemma}
%%%%%%%%%%%%%%%%%%%%%%%%%%%%%%%%%%%%%%%%%%%%%%%%%%%%%%%%%%%%%%%%%%%%%%%%%%%%%%%%%%%%%%%%%%%%%%%%%%%%
\noindent The first ingredient for step (i) is the following proposition.
%%%%%%%%%%%%%%%%%%%%%%%%%%%%%%%%%%%%%%%%%%%%%%%%%%%%%%%%%%%%%%%%%%%%%%%%%%%%%%%%%%%%%%%%%%%%%%%%%%%%
\begin{proposition}  \label{lem:discrete_stability1}
Let $1\leq j \leq N$ and suppose that the assumptions of Theorem~\ref{thm:convergence}{\rm{(a)}} are satisfied.
Then, the discrete functions $\left\{\sssh^{i+1}\right\}_{0\leq i \leq j-1}$ obtained through Algorithm~\ref{alg} fulfill
\begin{equation} \label{eq:stability1}
\norm{\sssh^j}{\LLL^2(\Omega)}^2
+ k \sum_{i=0}^{j-1}\norm{\sssh^{i+1}}{\HHH^1(\Omega)}^2
+ \sum_{i=0}^{j-1}\norm{\sssh^{i+1}-\sssh^i}{\LLL^2(\Omega)}^2
\leq C.
\end{equation}
The constant $C>0$ depends only on the data, but is in particular independent of the discretization parameters $h$ and $k$.
\end{proposition}
%%%%%%%%%%%%%%%%%%%%%%%%%%%%%%%%%%%%%%%%%%%%%%%%%%%%%%%%%%%%%%%%%%%%%%%%%%%%%%%%%%%%%%%%%%%%%%%%%%%%
\begin{proof}
Let $0\leq i \leq j-1$.
For~\eqref{eq:alg:3}, we choose $\zzeta_h= \sssh^{i+1}$ as test function.
After multiplication by $k$, we obtain
\begin{equation*}
\begin{split}
& \left(\sssh^{i+1}-\sssh^i, \sssh^{i+1} \right)_{\Omega}
+ k\left( D_0 \nabla \sssh^{i+1}, \nabla \sssh^{i+1} \right)_{\Omega}
- k\beta\beta' \left( D_0 \hmmmh^{i+1} \otimes \left(\nabla \sssh^{i+1}\cdot \hmmmh^{i+1} \right), \nabla \sssh^{i+1} \right)_{\Omega} \\
& \quad + k\left( D_0 \sssh^{i+1}, \sssh^{i+1} \right)_{\Omega}
%+ k (\underbrace{D_0\left(\sssh^{i+1} \times \hmmmh^{i+1}\right), \sssh^{i+1}}_{= \ 0} )_{\omega} \\
%& \qquad
= k \beta \left( \hmmmh^{i+1} \otimes \jjj^{i+1} , \nabla\sssh^{i+1} \right)_{\omega}
- k \beta \left(\jjj^{i+1}\cdot\nnn,\hmmmh^{i+1}\cdot\sssh^{i+1}\right)_{\partial\Omega\cap\partial\omega}.
\end{split}
\end{equation*}
Since $D_0 \geq D_*$ and $\norm{\hmmmh^{i+1}}{\LLL^{\infty}(\omega)}=1$, it follows that
\begin{equation*}
\begin{split}
&\left(\sssh^{i+1}-\sssh^i,\sssh^{i+1}\right)_{\Omega}
+ k D_* \left(1-\beta\beta'\right) \left(\nabla\sssh^{i+1},\nabla\sssh^{i+1}\right)_{\Omega}
+ k D_* \left(\sssh^{i+1},\sssh^{i+1}\right)_{\Omega} \\
& \quad \leq k \beta \left(\hmmmh^{i+1}\otimes\jjj^{i+1},\nabla\sssh^{i+1}\right)_{\omega}
- k \beta \left(\jjj^{i+1}\cdot\nnn,\hmmmh^{i+1}\cdot\sssh^{i+1}\right)_{\partial\Omega\cap\partial\omega},
\end{split}
\end{equation*}
cf.~the proof of Lemma~\ref{lem:coercivity}.
Summing up over $i=0,\dots,j-1$, and exploiting Abel's summation by parts from Lemma~\ref{lem:abel} for the term $\sum_{i=0}^{j-1} \left(\sssh^{i+1}-\sssh^i,\sssh^{i+1}\right)_{\Omega}$, we get
\begin{equation*}
\begin{split}
& \frac{1}{2} \norm{\sssh^j}{\LLL^2(\Omega)}^2
+ \frac{1}{2}\sum_{i=0}^{j-1} \norm{\sssh^{i+1}-\sssh^i}{\LLL^2(\Omega)}^2
+ k D_* \left(1-\beta\beta'\right) \sum_{i=0}^{j-1} \norm{\nabla\sssh^{i+1}}{\LLL^2(\Omega)}^2
+ k D_* \sum_{i=0}^{j-1} \norm{\sssh^{i+1}}{\LLL^2(\Omega)}^2 \\
& \quad \leq \frac{1}{2}\norm{\sssh^0}{\LLL^2(\Omega)}^2
+ k \beta \sum_{i=0}^{j-1} \left[\left(\hmmmh^{i+1}\otimes\jjj^{i+1},\nabla\sssh^{i+1}\right)_{\omega}
- \left(\jjj^{i+1}\cdot\nnn,\hmmmh^{i+1}\cdot\sssh^{i+1}\right)_{\partial\Omega\cap\partial\omega}\right].
\end{split}
\end{equation*}
Exploiting $0<1-\beta\beta'<1$ on the left-hand side, the Cauchy-Schwarz inequality and the Young inequality on the right-hand side, we obtain, for any choice of $\eps>0$,
\begin{equation*}
\begin{split}
& \frac{1}{2} \norm{\sssh^j}{\LLL^2(\Omega)}^2
+ \frac{1}{2}\sum_{i=0}^{j-1} \norm{\sssh^{i+1}-\sssh^i}{\LLL^2(\Omega)}^2
+ k D_* \left(1-\beta\beta'\right) \sum_{i=0}^{j-1} \norm{\sssh^{i+1}}{\HHH^1(\Omega)}^2 \\
& \quad \leq \frac{1}{2} \norm{\sssh^0}{\LLL^2(\Omega)}^2
+ \frac{Ck \beta}{2 \eps}\sum_{i=0}^{j-1}\norm{\jjj^{i+1}}{\HHH^1(\Omega)}^2
+ \frac{Ck \beta \eps}{2}\sum_{i=0}^{j-1} \norm{\sssh^{i+1}}{\HHH^1(\Omega)}^2.
\end{split}
\end{equation*}
Here the constant $C>0$ is the stability constant of the trace operator.
It follows that
\begin{equation*}
\begin{split}
& \frac{1}{2} \norm{\sssh^j}{\LLL^2(\Omega)}^2
+ \frac{1}{2}\sum_{i=0}^{j-1} \norm{\sssh^{i+1}-\sssh^i}{\LLL^2(\Omega)}^2
+ k \left[ D_* \left(1-\beta\beta'\right)-\frac{C \beta \eps}{2}\right] \sum_{i=0}^{j-1} \norm{\sssh^{i+1}}{\HHH^1(\Omega)}^2 \\
& \quad \leq \frac{1}{2} \norm{\sssh^0}{\LLL^2(\Omega)}^2
+ \frac{C k \beta}{2 \eps}\sum_{i=0}^{j-1} \norm{\jjj^{i+1}}{\HHH^1(\Omega)}^2.
\end{split}
\end{equation*}
If we choose $\eps<2D_*(1-\beta\beta')/C\beta$, then all the coefficients on the left-hand side are positive.
From~(H1) and the regularity of $\jjj$, we know that the right-hand side is uniformly bounded with respect to $h$ and $k$.
This yields the estimate~\eqref{eq:stability1}.
\end{proof}
%%%%%%%%%%%%%%%%%%%%%%%%%%%%%%%%%%%%%%%%%%%%%%%%%%%%%%%%%%%%%%%%%%%%%%%%%%%%%%%%%%%%%%%%%%%%%%%%%%%%
\begin{corollary} \label{lem:boundedness3}
Under the assumptions of Proposition~\ref{lem:discrete_stability1}, the sequences $\left\{\ssshk\right\}$ and $\left\{\ssshk^{\pm}\right\}$ are uniformly bounded in $L^2(0,T;\HHH^1(\Omega))$ and in $L^{\infty}(0,T;\LLL^2(\Omega))$, i.e.,
\begin{equation*}
\norm{\ssshk}{L^2(0,T;\HHH^1(\Omega))}
+ \norm{\ssshk^{\pm}}{L^2(0,T;\HHH^1(\Omega))}
+ \norm{\ssshk}{L^{\infty}(0,T;\LLL^2(\Omega))}
+ \norm{\ssshk^{\pm}}{L^{\infty}(0,T;\LLL^2(\Omega))}
\leq C,
\end{equation*}
where the constant $C>0$ depends only on the data, but is in particular independent of the discretization parameters $h$ and $k$.
\end{corollary}
%%%%%%%%%%%%%%%%%%%%%%%%%%%%%%%%%%%%%%%%%%%%%%%%%%%%%%%%%%%%%%%%%%%%%%%%%%%%%%%%%%%%%%%%%%%%%%%%%%%%
\begin{proof}
The result follows from the boundedness of the discrete functions
$\left\{\sssh^{i+1}\right\}_{0 \leq i \leq N-1}$ from Proposition~\ref{lem:discrete_stability1}.
\end{proof}
%%%%%%%%%%%%%%%%%%%%%%%%%%%%%%%%%%%%%%%%%%%%%%%%%%%%%%%%%%%%%%%%%%%%%%%%%%%%%%%%%%%%%%%%%%%%%%%%%%%%
\noindent Let $\PP_h:\LLL^2(\Omega)\to\SS^1(\TT_h^{\Omega})^3$ be the $\LLL^2(\Omega)$-orthogonal projection onto $\SS^1(\TT_h^{\Omega})^3$, i.e.,
\begin{equation*}
\left(\PP_h \eeta-\eeta,\eeta_h\right)_{\Omega}=0 \quad \text{for all } \eeta \in \LLL^2(\Omega), \ \eeta_h \in \SS^1(\TT_h^{\Omega})^3.
\end{equation*}
Since $\left\{\TT_h^{\Omega}\right\}_{h>0}$ is quasi-uniform, it is well known that $\PP_h$ is stable in $\HHH^1(\Omega)$, i.e.,
\begin{equation} \label{eq:l2_projection_h1_stability}
\norm{\PP_h \eeta}{\HHH^1(\Omega)} \lesssim \norm{\eeta}{\HHH^1(\Omega)} \quad \text{for all } \eeta \in \HHH^1(\Omega).
\end{equation}
We also refer to \cite{by2014,kpp2013} for recent results on $H^1$-stability on locally refined meshes.
With this, we obtain uniform boundedness of $\partial_t\ssshk$.
%%%%%%%%%%%%%%%%%%%%%%%%%%%%%%%%%%%%%%%%%%%%%%%%%%%%%%%%%%%%%%%%%%%%%%%%%%%%%%%%%%%%%%%%%%%%%%%%%%%%
\begin{proposition} \label{lem:boundedness2}
The sequence $\left\{\partial_t\ssshk\right\}$ is uniformly bounded in $L^2(0,T;\wtd\HHH^{-1}(\Omega))$, i.e.,
\begin{equation} \label{eq:boundedness2}
\norm{\partial_t\ssshk}{L^2(0,T;\wtd\HHH^{-1}(\Omega))} \leq C,
\end{equation}
where the constant $C>0$ depends only on the data, but is in particular independent of the discretization parameters $h$ and $k$.
\end{proposition}
%%%%%%%%%%%%%%%%%%%%%%%%%%%%%%%%%%%%%%%%%%%%%%%%%%%%%%%%%%%%%%%%%%%%%%%%%%%%%%%%%%%%%%%%%%%%%%%%%%%%
\begin{proof}
Let $\www \in \HHH^1(\Omega)\setminus\left\{\zero\right\}$, $0\leq i \leq N-1$, and $t \in [t_i,t_{i+1})$.
From~\eqref{eq:alg:3} and the $H^1$-stability~\eqref{eq:l2_projection_h1_stability} of $\PP_h$, we get
\begin{equation*}
\begin{split}
\left\langle\partial_t \ssshk(t),\www\right\rangle
& = \left(\partial_t \ssshk(t),\www\right)_{\Omega}
= \left(\dt\sssh^{i+1},\www\right)_{\Omega}
= \left(\dt\sssh^{i+1},\PP_h\www\right)_{\Omega} \\
& = \beta \left( \hmmmh^{i+1} \otimes \jjj^{i+1} , \nabla\PP_h\www \right)_{\omega}
- \beta \left(\jjj^{i+1}\cdot\nnn,\hmmmh^{i+1}\cdot\PP_h\www\right)_{\partial\Omega\cap\partial\omega}
- a_h^{i+1}(\sssh^{i+1},\PP_h\www) \\
& \lesssim \left(\norm{\jjj^{i+1}}{\HHH^1(\Omega)} + \norm{\sssh^{i+1}}{\HHH^1(\Omega)}\right) \norm{\PP_h\www}{\HHH^1(\Omega)}
\lesssim \left(\norm{\jjj^{i+1}}{\HHH^1(\Omega)} + \norm{\sssh^{i+1}}{\HHH^1(\Omega)}\right) \norm{\www}{\HHH^1(\Omega)}.
\end{split}
\end{equation*} 
Dividing by $\norm{\www}{\HHH^1(\Omega)}$ and taking the supremum over $\www \in \HHH^1(\Omega)\setminus\left\{\zero\right\}$, we obtain
\begin{equation*}
\norm{\partial_t\ssshk(t)}{\wtd\HHH^{-1}(\Omega)} \lesssim \norm{\jjj^{i+1}}{\HHH^1(\Omega)} + \norm{\sssh^{i+1}}{\HHH^1(\Omega)}.
\end{equation*}
Squaring, integrating over $(t_i,t_{i+1})$, and summing up over $0\leq i \leq N-1$, we get
\begin{equation*}
\norm{\partial_t\ssshk}{L^2(0,T;\wtd\HHH^{-1}(\Omega))}^2 \lesssim k \sum_{i=0}^{N-1}\norm{\jjj^{i+1}}{\HHH^1(\Omega)}^2 + k \sum_{i=0}^{N-1} \norm{\sssh^{i+1}}{\HHH^1(\Omega)}^2.
\end{equation*}
The boundedness from Proposition~\ref{lem:discrete_stability1} thus yields~\eqref{eq:boundedness2}.
\end{proof}
%%%%%%%%%%%%%%%%%%%%%%%%%%%%%%%%%%%%%%%%%%%%%%%%%%%%%%%%%%%%%%%%%%%%%%%%%%%%%%%%%%%%%%%%%%%%%%%%%%%%
\noindent We derive the corresponding estimates for the discrete quantities $\left\{\left(\vvv_h^i,\mmmh^{i+1}\right)\right\}_{0\leq i \leq N-1}$.
%%%%%%%%%%%%%%%%%%%%%%%%%%%%%%%%%%%%%%%%%%%%%%%%%%%%%%%%%%%%%%%%%%%%%%%%%%%%%%%%%%%%%%%%%%%%%%%%%%%%
\begin{lemma} \label{lem:new_stability}
Let $0\leq i \leq N-1$. The discrete functions $\left(\vvvh^i,\mmmh^{i+1}\right)$ obtained through Algorithm~\ref{alg} fulfill
\begin{equation} \label{eq:new_stability}
\begin{split}
& \alpha \norm{\vvvh^i}{\LLL^2(\omega)}^2
+ \frac{\Ce}{2k} \left(\norm{\nabla\mmmh^{i+1}}{\LLL^2(\omega)}^2
- \norm{\nabla\mmmh^i}{\LLL^2(\omega)}^2\right)
+ \Ce k \left(\theta - \frac{1}{2}\right) \norm{\nabla \vvvh^i}{\LLL^2(\omega)}^2 \\
& \quad = \left(\ppih(\mmmh^i),\vvvh^i\right)_{\omega}
+ \left(\fff^i,\vvvh^i\right)_{\omega}
+ c \left(\sssh^i,\vvvh^i\right)_{\omega}.
\end{split}
\end{equation}
\end{lemma}
%%%%%%%%%%%%%%%%%%%%%%%%%%%%%%%%%%%%%%%%%%%%%%%%%%%%%%%%%%%%%%%%%%%%%%%%%%%%%%%%%%%%%%%%%%%%%%%%%%%%
\begin{proof}
We test~\eqref{eq:alg:1} with $\pphi_h = \vvvh^i \in \KK_{\mmmh^i}$ to get
% \begin{equation*}
% \begin{split}
% & \alpha \norm{\vvvh^i}{\LLL^2(\omega)}^2
% + (\overbrace{\mmmh^i \times \vvvh^i,\vvv_h^i}^{= \ 0})_{\omega}
% + \Ce \theta k \norm{\nabla \vvvh^i}{\LLL^2(\omega)}^2 \\
% & \quad = - \Ce \left(\nabla \mmmh^i, \nabla \vvvh^i\right)_{\omega}
% + \left(\ppih(\mmmh^i),\vvvh^i\right)_{\omega}
% + \left(\fff^i,\vvvh^i\right)_{\omega}
% + c \left(\sssh^i,\vvvh^i\right)_{\omega},
% \end{split}
% \end{equation*}
% which is equivalent to
\begin{equation*}
\begin{split}
& \alpha \norm{\vvvh^i}{\LLL^2(\omega)}^2
+ \frac{\Ce}{k} \left(\nabla \mmmh^i+k\nabla \vvvh^i,k\nabla \vvvh^i\right)_{\omega}
+ \Ce k \left(\theta - 1\right) \norm{\nabla \vvvh^i}{\LLL^2(\omega)}^2 \\
& \quad = \left(\ppih(\mmmh^i),\vvvh^i\right)_{\omega}
+ \left(\fff^i,\vvvh^i\right)_{\omega}
+ c \left(\sssh^i,\vvvh^i\right)_{\omega}.
\end{split}
\end{equation*}
Exploiting the vector identity
\begin{equation*}
2\left(\aaa+\bbb\right)\cdot\aaa=\abs{\aaa}^2 + \abs{\aaa+\bbb}^2 - \abs{\bbb}^2 \quad \text{for all } \aaa,\bbb\in\R^3
\end{equation*}
with the choice $\aaa=k \nabla \vvvh^i$ and $\bbb=\nabla\mmmh^i$, and taking into account~\eqref{eq:alg:2}, we obtain~\eqref{eq:new_stability}.
\end{proof}
%%%%%%%%%%%%%%%%%%%%%%%%%%%%%%%%%%%%%%%%%%%%%%%%%%%%%%%%%%%%%%%%%%%%%%%%%%%%%%%%%%%%%%%%%%%%%%%%%%%%
\begin{proposition} \label{lem:discrete_stability2}
Suppose that the assumptions of Theorem~\ref{thm:convergence}{\rm{(a)}} are satisfied.
Then, there exists $k_0>0$ such that for all time-step sizes $0<k<k_0$ and $1 \leq j \leq N$ the discrete functions $\left\{\left(\vvv_h^i,\mmmh^{i+1}\right)\right\}_{0\leq i \leq j-1}$ obtained through Algorithm~\ref{alg} fulfill
\begin{equation} \label{eq:stability2}
\norm{\nabla\mmmh^j}{\LLL^2(\omega)}^2
+ k \sum_{i=0}^{j-1}\norm{\vvv_h^i}{\LLL^2(\omega)}^2
+ \left(\theta-\frac{1}{2}\right)k^2 \sum_{i=0}^{j-1}\norm{\nabla\vvv_h^i}{\LLL^2(\omega)}^2
\leq C.
\end{equation}
The constant $C>0$ depends only on the data and $k_0$, but is otherwise independent of the discretization parameters $h$ and $k$.
\end{proposition}
%%%%%%%%%%%%%%%%%%%%%%%%%%%%%%%%%%%%%%%%%%%%%%%%%%%%%%%%%%%%%%%%%%%%%%%%%%%%%%%%%%%%%%%%%%%%%%%%%%%%
\begin{proof}
Let $1 \leq j \leq N$.
From Lemma~\ref{lem:new_stability}, multiplying~\eqref{eq:new_stability} by $k/\Ce$, summing up over $0 \leq i \leq j-1$ and exploiting the telescopic sum, we obtain
\begin{equation*}
\begin{split}
& \frac{1}{2} \norm{\nabla\mmmh^j}{\LLL^2(\omega)}^2
+ \frac{\alpha k}{\Ce} \sum_{i=0}^{j-1} \norm{\vvvh^i}{\LLL^2(\omega)}^2
+ k^2 \left(\theta - \frac{1}{2}\right) \sum_{i=0}^{j-1} \norm{\nabla \vvvh^i}{\LLL^2(\omega)}^2 \\
& \quad = \frac{1}{2} \norm{\nabla\mmmh^0}{\LLL^2(\omega)}^2
+ \frac{k}{\Ce} \sum_{i=0}^{j-1} \left[ \left(\ppih(\mmmh^i),\vvvh^i \right)_{\omega} + \left(\fff^i,\vvvh^i\right)_{\omega} + c \left(\sssh^i,\vvvh^i\right)_{\omega} \right].
\end{split}
\end{equation*}
The Cauchy-Schwarz inequality and the Young inequality, together with assumption~(H2), yield for any $\eps>0$
\begin{equation*}
\begin{split}
&\frac{1}{2} \norm{\nabla \mmmh^j}{\LLL^2(\omega)}^2
+ \frac{k}{\Ce} \left(\alpha-\frac{2+c}{2}\eps\right) \sum_{i=0}^{j-1}\norm{\vvv_h^i}{\LLL^2(\omega)}^2
+ k^2 \left(\theta-\frac{1}{2}\right) \sum_{i=0}^{j-1} \norm{\nabla \vvv_h^i}{\LLL^2(\omega)}^2 \\
& \quad \leq \frac{1}{2}\norm{\nabla \mmmh^0}{\LLL^2(\omega)}^2
+ \frac{k}{2 \eps \Ce} \sum_{i=0}^{j-1} \left[ C_{\ppi}^2 \norm{\mmmh^i}{\LLL^2(\omega)}^2 + \norm{\fff^i}{\LLL^2(\omega)}^2 + c \norm{\sssh^i}{\LLL^2(\omega)}^2 \right].
\end{split}
\end{equation*}
From Proposition~\ref{prop:discrete_wellposedness}, we deduce
\begin{equation*}
k \sum_{i=0}^{j-1} \norm{\mmmh^i}{\LLL^2(\omega)}^2
\leq C' \left(1 + k^2 \sum_{i=0}^{j-1} \norm{\vvvh^i}{\LLL^2(\omega)}^2\right),
\end{equation*}
where the constant $C'>0$ depends only on $\abs{\omega}$, $T$ and $C_*$.
We thus obtain
\begin{equation*}
\begin{split}
&\frac{1}{2} \norm{\nabla \mmmh^j}{\LLL^2(\omega)}^2
+ \frac{k}{\Ce} \left(\alpha-\frac{2+c}{2}\eps-\frac{k C_{\ppi}^2 C'}{2 \eps}\right) \sum_{i=0}^{j-1}\norm{\vvv_h^i}{\LLL^2(\omega)}^2
+ k^2 \left(\theta-\frac{1}{2}\right) \sum_{i=0}^{j-1} \norm{\nabla \vvv_h^i}{\LLL^2(\omega)}^2 \\
& \quad \leq \frac{1}{2}\norm{\nabla \mmmh^0}{\LLL^2(\omega)}^2
+ \frac{C'C_{\ppi}^2}{2 \eps \Ce}
+ \frac{k}{2 \eps \Ce} \sum_{i=0}^{j-1} \left[\norm{\fff^i}{\LLL^2(\omega)}^2 + c \norm{\sssh^i}{\LLL^2(\omega)}^2 \right].
\end{split}
\end{equation*}
Note that $\theta>1/2$. If we choose $\eps<2\alpha/(2+c)$, for $k<k_0:=\eps\left(2 \alpha - (2+c)\eps\right)/\left(C_{\ppi}^2 C'\right)$ all the coefficients on the left-hand side are positive.
From the regularity of $\fff$, assumption~(H1), and the boundedness from Proposition~\ref{lem:discrete_stability1}, we know that the right-hand side is uniformly bounded.
This yields the estimate~\eqref{eq:stability2}.
\end{proof}
%%%%%%%%%%%%%%%%%%%%%%%%%%%%%%%%%%%%%%%%%%%%%%%%%%%%%%%%%%%%%%%%%%%%%%%%%%%%%%%%%%%%%%%%%%%%%%%%%%%%
\begin{corollary} \label{lem:boundedness}
Under the assumptions of Proposition~\ref{lem:discrete_stability2}, and if $k < k_0$, the sequences $\left\{\mmmhk\right\}$, $\left\{\mmmhk^{\pm}\right\}$, $\left\{\Pi_h\mmmhk^+\right\}$ and $\left\{\vvvhk^{-}\right\}$ are uniformly bounded.
In particular, it holds
\begin{equation*}
\norm{\mmmhk}{\HHH^1(\omega_T)}
+ \norm{\mmmhk^{\pm}}{L^2(0,T;\HHH^1(\omega))}
+ \norm{\Pi_h\mmmhk^+}{L^2(0,T;\HHH^1(\omega))}
+ \norm{\vvvhk^{-}}{\LLL^2(\omega_T)} \leq C,
\end{equation*}
where the constant $C>0$ depends only on the data and $k_0$, but is independent of the discretization parameters $h$ and $k$.
\end{corollary}
%%%%%%%%%%%%%%%%%%%%%%%%%%%%%%%%%%%%%%%%%%%%%%%%%%%%%%%%%%%%%%%%%%%%%%%%%%%%%%%%%%%%%%%%%%%%%%%%%%%%
\begin{proof}
The result follows from the boundedness of the discrete functions $\left\{\left(\vvv_h^i,\mmmh^{i+1}\right)\right\}_{0\leq i \leq N-1}$ from Proposition~\ref{prop:discrete_wellposedness} and Proposition~\ref{lem:discrete_stability2}, and from~\eqref{eq:energyDecay}.
\end{proof}
%%%%%%%%%%%%%%%%%%%%%%%%%%%%%%%%%%%%%%%%%%%%%%%%%%%%%%%%%%%%%%%%%%%%%%%%%%%%%%%%%%%%%%%%%%%%%%%%%%%%
\noindent We can now proceed with step~(ii) of the proof and conclude the existence of weakly convergent subsequences.
%%%%%%%%%%%%%%%%%%%%%%%%%%%%%%%%%%%%%%%%%%%%%%%%%%%%%%%%%%%%%%%%%%%%%%%%%%%%%%%%%%%%%%%%%%%%%%%%%%%%
\begin{proposition}\label{lem:subsequences}
Suppose that the assumptions of Theorem~\ref{thm:convergence}{{\rm(a)}} are satisfied.
Then, there exist $\mmm\in\HHH^1(\omega_T) \cap L^{\infty}(0,T;\HHH^1(\omega))$ and $\sss\in L^2(0,T;\HHH^1(\Omega)) \cap L^{\infty}(0,T;\LLL^2(\Omega)) \cap H^1(0,T;\wtd\HHH^{-1}(\Omega))$, with $\abs{\mmm}=1$ a.e.\ in $\omega_T$, such that there holds
\begin{subequations}\label{eq:subsequences}
\begin{align}
\label{eq:subsequences1} \mmmhk \subweakto \mmm & \text{ in } \HHH^1(\omega_T),\\
\mmmhk, \mmmhk^\pm, \Pi_h\mmmhk^+ \subweakto \mmm & \text{ in } L^2(0,T;\HHH^1(\omega)), \\
\mmmhk, \mmmhk^\pm, \Pi_h\mmmhk^+ \subto \mmm & \text{ in } \LLL^2(\omega_T),\\
\label{eq:subsequences4} \vvvhk^- \subweakto \mmmt & \text{ in } \LLL^2(\omega_T),\\
\label{eq:subsequences5} \ssshk, \ssshk^\pm \subweakto \sss & \text{ in } L^2(0,T;\HHH^1(\Omega)),\\
\label{eq:subsequences6} \partial_t \ssshk \subweakto \ssst & \text{ in } L^2(0,T;\wtd\HHH^{-1}(\Omega))
\end{align}
\end{subequations}
for $(h,k) \to (0,0)$.
Moreover, there exists one subsequence for which~\eqref{eq:subsequences} holds simultaneously.
\end{proposition}
%%%%%%%%%%%%%%%%%%%%%%%%%%%%%%%%%%%%%%%%%%%%%%%%%%%%%%%%%%%%%%%%%%%%%%%%%%%%%%%%%%%%%%%%%%%%%%%%%%%%
\begin{proof}
The boundedness results from Corollary~\ref{lem:boundedness}, in combination with the Eberlein-Smulian theorem, allow us to extract weakly convergent subsequences of $\left\{\mmmhk\right\}$, $\left\{\mmmhk^{\pm}\right\}$, $\left\{\Pi_h\mmmhk^+\right\}$ and $\left\{\vvvhk^-\right\}$.
Let $\mmm\in\HHH^1(\omega_T)$ be such that $\mmmhk\subweakto\mmm$ in $\HHH^1(\omega_T)$.
From the continuous inclusions $\HHH^1(\omega_T) \subset L^2(0,T;\HHH^1(\omega)) \subset \LLL^2(\omega_T)$ and the compact embedding $\HHH^1(\omega_T) \Subset \LLL^2(\omega_T)$, we deduce
\begin{equation*}
\mmmhk \subweakto \mmm \text{ in } L^2(0,T;\HHH^1(\omega)) \quad \text{and} \quad \mmmhk \subto \mmm \text{ in } \LLL^2(\omega_T).
\end{equation*}
With $\norm{\mmmhk-\mmmhk^{\pm}}{\LLL^2(\omega_T)} \leq k \norm{\vvvhk^-}{\LLL^2(\omega_T)}$, we can identify the limits of the subsequences of $\left\{\mmmhk\right\}$ and $\left\{\mmmhk^{\pm}\right\}$.
As $\partial_t\mmmhk=\vvvhk^-$, it clearly holds that $\vvv=\mmmt$ a.e.\ in $\omega_T$.
\par We now prove that the limiting function $\mmm$ satisfies the unit-length constraint.
First, we observe that
\begin{equation} \label{eq:unit1}
\norm{\abs{\mmm}^2-\abs{\mmmhk^+}^2}{L^1(\omega_T)} \leq \norm{\mmm+\mmmhk^+}{\LLL^2(\omega_T)} \norm{\mmm-\mmmhk^+}{\LLL^2(\omega_T)} \subto 0
\end{equation}
for $(h,k)\to 0$.
For $0 \leq i \leq N-1$ and $K \in \TT_h^{\omega}$, a standard interpolation estimate for the piecewise linear function $\mmmh^{i+1} \in \SS^1(\TT_h^{\omega})^3$ yields
\begin{equation*}
\norm{\abs{\mmmh^{i+1}}^2-\II_h\left(\abs{\mmmh^{i+1}}^2\right)}{L^2(K)}
\lesssim h_K^2 \norm{D^2\abs{\mmmh^{i+1}}^2}{L^2(K)}
\lesssim h_K^2 \norm{\nabla\mmmh^{i+1}}{\LLL^4(K)}^2
\lesssim h_K^{1/2} \norm{\nabla\mmmh^{i+1}}{\LLL^2(K)}^2.
\end{equation*}
From Proposition~\ref{lem:discrete_stability2}, we obtain
\begin{equation} \label{eq:unit2}
\norm{\abs{\mmmhk^+}^2-\II_h\left(\abs{\mmmhk^+}^2\right)}{L^2(\omega_T)} \lesssim h^{1/2}.
\end{equation}
For all $1 \leq j \leq N$, Proposition~\ref{prop:discrete_wellposedness} and the discrete norm equivalence of Lemma~\ref{lem:norm_equiv} with $r=2$ yield
\begin{equation*}
\norm{\II_h\left( \abs{\mmmh^j}^2 \right)-1}{L^1(\omega)}
\lesssim \sum_{\zzz \in \NN_h^{\omega}} h^3 \abs{\abs{\mmmh^j(\zzz)}^2 - 1}
\leq k^2 \sum_{i=0}^{j-1} h^3 \sum_{\zzz \in \NN_h^{\omega}} \abs{\vvvh^{i}(\zzz)}^2
\lesssim k^2 \sum_{i=0}^{j-1} \norm{\vvvh^i}{\LLL^2(\omega)}^2.
\end{equation*}
Then, from Proposition~\ref{lem:discrete_stability2}, we deduce
\begin{equation} \label{eq:unit3}
\norm{\II_h\left(\abs{\mmmhk^+}^2\right)-1}{L^1(\omega_T)} \lesssim k.
\end{equation}
Combining~\eqref{eq:unit2}--\eqref{eq:unit3}, the triangle inequality thus yields that $\abs{\mmmhk^+}^2 \subto 1$ in $\LLL^1(\omega_T)$ for $(h,k)\to 0$, whence $\abs{\mmm}=1$ a.e.\ in $\omega_T$ follows from~\eqref{eq:unit1}.
\par For $\xxx \in \R^3$ with $\abs{\xxx}\geq 1$, it holds that
\begin{equation*}
\abs{\xxx-\frac{\xxx}{\abs{\xxx}}} = \abs{\xxx}-1 = \frac{\abs{\xxx}^2-1}{\abs{\xxx}+1} \leq \frac{1}{2}\left(\abs{\xxx}^2-1\right).
\end{equation*}
Due to~\eqref{eq:nodewise}, this yields for all $1 \leq j \leq N$
\begin{equation*}
\abs{\mmmh^j(\zzz)-\Pi_h\mmmh^j(\zzz)} \leq \frac{1}{2}\left(\abs{\mmmh^j(\zzz)}^2 -1 \right) = \frac{1}{2} k^2 \sum_{i=0}^{j-1} \abs{\vvvh^i(\zzz)}^2,
\end{equation*}
whence by virtue of Proposition~\ref{lem:discrete_stability2}
\begin{equation*}
\norm{\mmmh^j-\Pi_h\mmmh^j}{\LLL^1(\omega)}
\lesssim k^2 \sum_{i=0}^{j-1} \norm{\vvvh^j}{\LLL^2(\omega)}^2
\lesssim k.
\end{equation*}
This implies $\Pi_h\mmmhk^+ \subto \mmm$ in $\LLL^1(\omega_T)$ as $(h,k) \to (0,0)$.
Since $\norm{\Pi_h\mmmhk^+}{\LLL^{\infty}(\omega_T)}+\norm{\mmm}{\LLL^{\infty}(\omega_T)}=2$, we have $\Pi_h\mmmhk^+ \subto \mmm$ even in $\LLL^2(\omega_T)$ as well as $\Pi_h\mmmhk^+ \subweakto \mmm$ in $L^2(0,T;\HHH^1(\omega))$.
\par From Corollary~\ref{lem:boundedness3}, we similarly deduce the existence of weakly convergent subsequences of $\left\{\ssshk\right\}$ and $\left\{\ssshk^{\pm}\right\}$.
Due to Proposition~\ref{lem:discrete_stability1}, the quantity $\sum_{i=0}^{N-1} \norm{\sssh^{i+1}-\sssh^i}{\LLL^2(\Omega)}^2$ is bounded.
This allows to identify the weak limits, since
\begin{equation*}
\norm{\ssshk-\ssshk^{\pm}}{\LLL^2(\Omega_T)} \lesssim k \sum_{i=0}^{N-1} \norm{\sssh^{i+1}-\sssh^i}{\LLL^2(\Omega)}^2 \longrightarrow 0
\quad \text{ for } (h,k) \to 0.
\end{equation*}
Finally, from Proposition~\ref{lem:boundedness2}, we deduce the existence of a weakly convergent subsequence of $\left\{\partial_t\ssshk\right\}$, and it is easy to see that its limit is precisely $\ssst$, cf.~\cite[Section~7.1.2, Theorem~3]{evans2010}.
This establishes~\eqref{eq:subsequences5}--\eqref{eq:subsequences6} and thus concludes the proof.
\end{proof}
%%%%%%%%%%%%%%%%%%%%%%%%%%%%%%%%%%%%%%%%%%%%%%%%%%%%%%%%%%%%%%%%%%%%%%%%%%%%%%%%%%%%%%%%%%%%%%%%%%%%
\begin{remark}
As the constants which guarantee the boundedness of Proposition~\ref{lem:discrete_stability1}, Corollary~\ref{lem:boundedness3} and Proposition~\ref{lem:boundedness2}, are independent of $T$, we deduce that $\sss \in L^2(\R^+;\HHH^1(\Omega))\cap L^2(\R^+;\LLL^2(\Omega))\cap H^1(\R^+;\wtd\HHH^{-1}(\Omega))$.
\end{remark}
%%%%%%%%%%%%%%%%%%%%%%%%%%%%%%%%%%%%%%%%%%%%%%%%%%%%%%%%%%%%%%%%%%%%%%%%%%%%%%%%%%%%%%%%%%%%%%%%%%%%
\noindent We have collected all the ingredients for the proof of our main theorem.
%%%%%%%%%%%%%%%%%%%%%%%%%%%%%%%%%%%%%%%%%%%%%%%%%%%%%%%%%%%%%%%%%%%%%%%%%%%%%%%%%%%%%%%%%%%%%%%%%%%%
\begin{proof}[Proof of Theorem~\ref{thm:convergence}]
The result of part~(a) follows directly from Proposition~\ref{lem:subsequences}.
To conclude the proof of part~(b), it remains to identify the limiting functions $(\mmm,\sss)$ with a weak solution of SDLLG in the sense of Definition~\ref{def:weak_sol}.
\par To check~\eqref{eq:weak_spin_llg}, we essentially proceed as in~\cite{alouges2008a}.
Let $\vphi \in \CCC^\infty(\overline{\omega_T})$.
For $0 \leq i \leq N-1$,  we test~\eqref{eq:alg:1} with respect to $\pphi_h = \II_h \left(\left(\mmmhk^-\times \vphi \right)(t_i)\right) \in \KK_{\mmmh^i}$ , with $\II_h$ being the nodal interpolation operator onto $\SS^1(\TT_h^{\omega})^3$.
Multiplication with $k$ and summation over $0 \leq i \leq N-1$ yield
\begin{equation*}
\begin{split}
& \left(\alpha\vvvhk^- + \mmmhk^- \times \vvvhk^-,\II_h \left(\mmmhk^-\times \vphi_k^-\right) \right)_{\omega_T} \\
& \quad = -\Ce \left(\nabla\left(\mmmhk^- + \theta k \vvvhk^- \right), \nabla\II_h \left(\mmmhk^-\times \vphi_k^-\right) \right)_{\omega_T}
+ \left( \ppih(\mmmhk^-), \II_h \left(\mmmhk^-\times \vphi_k^-\right) \right)_{\omega_T} \\
& \qquad + \left( \fff_k^-, \II_h \left(\mmmhk^-\times \vphi_k^-\right) \right)_{\omega_T}
+ c \left( \ssshk^-, \II_h \left(\mmmhk^-\times \vphi_k^-\right) \right)_{\omega_T},
\end{split}
\end{equation*}
where $\II_h \left(\mmmhk^-\times \vphi_k^-\right)(t)=\II_h \left(\left(\mmmhk^-\times \vphi \right)(t_i)\right)$ for all $t \in [t_i,t_{i+1})$.
From the well-known approximation properties of $\II_h$ and the boundedness of $\sqrt{k} \norm{\nabla\vvvhk^-}{\LLL^2(\omega_T)}$ from Proposition~\ref{lem:discrete_stability2} for $\theta \in (1/2,1]$, we deduce
\begin{equation*}
\begin{split}
& \left(\alpha\vvvhk^- + \mmmhk^- \times \vvvhk^-,\mmmhk^-\times \vphi_k^- \right)_{\omega_T}
+ \Ce \left(\nabla\left(\mmmhk^- + \theta k \vvvhk^- \right), \nabla\left(\mmmhk^-\times \vphi_k^-\right) \right)_{\omega_T} \\
& \quad - \left( \ppih(\mmmhk^-), \mmmhk^-\times \vphi_k^- \right)_{\omega_T}
- \left( \fff_k^-, \mmmhk^-\times \vphi_k^- \right)_{\omega_T}
- c \left( \ssshk^-, \mmmhk^-\times \vphi_k^- \right)_{\omega_T}
= \bigO{h}.
\end{split}
\end{equation*}
Passing to the limit for $(h,k) \to (0,0)$, we obtain
\begin{equation*}
\left(\alpha\mmmt + \mmm \times \mmmt,\mmm \times \vphi \right)_{\omega_T}
= - \Ce \left(\nabla\mmm,\nabla\left(\mmm\times\vphi\right) \right)_{\omega_T}
+ \left(\ppi(\mmm)+\fff+c\sss,\mmm\times\vphi\right)_{\omega_T}
\end{equation*}
In the latter, we have used the convergence properties from Proposition~\ref{lem:subsequences}, assumption~(H3) for the general field contribution, as well as $\fff_k^-\weakto\fff$ and $\mmmhk^-\times \vphi_k^- \subto \mmm\times\vphi$ in $\LLL^2(\omega_T)$.
\par Direct calculations and standard properties of the cross product yield the identities
\begin{gather*}
\left(\nabla\mmm,\nabla(\mmm\times\vphi) \right)_{\omega_T} = \left(\nabla\mmm\times\mmm,\nabla\vphi \right)_{\omega_T},
\quad \left(\mmmt,\mmm \times \vphi \right)_{\omega_T} = \left(\mmmt \times \mmm, \vphi \right)_{\omega_T}, \\
\left(\mmm \times \mmmt,\mmm \times \vphi \right)_{\omega_T} = \left(\mmmt,\vphi \right)_{\omega_T},
\quad \left(\ppi(\mmm),\mmm\times\vphi\right)_{\omega_T} = \left(\ppi(\mmm)\times\mmm,\vphi\right)_{\omega_T}, \\
\left(\fff,\mmm\times\vphi\right)_{\omega_T} = \left(\fff\times\mmm,\vphi\right)_{\omega_T},
\quad \left(\sss,\mmm\times\vphi\right)_{\omega_T} = \left(\sss\times\mmm,\vphi\right)_{\omega_T},
\end{gather*}
from which, by density, we deduce~\eqref{eq:weak_spin_llg}.
\par To check~\eqref{eq:weak_spin_diff}, let $\vphi \in C^{\infty}(0,T;\CCC^{\infty}(\overline\Omega))$.
Given $0 \leq i \leq N-1$, let $t \in [t_i,t_{i+1})$.
In~\eqref{eq:alg:3} we choose the test function $\zzeta_h = \II_h \left(\vphi(t)\right) \in \SS^1(\TT_h^{\Omega})^3$.
Integration in time over $(t_i,t_{i+1})$ and summation over $0 \leq i \leq N-1$ yield
\begin{equation} \label{eq:last}
\begin{split}
& \left( \partial_t \ssshk,\II_h\vphi \right)_{\Omega_T}
+ \left( D_0 \nabla \ssshk^+, \nabla \II_h\vphi \right)_{\Omega_T}
- \beta\beta' \left( D_0 \Pi_h\mmmhk^+ \otimes \left(\nabla \ssshk^+\cdot \Pi_h\mmmhk^+ \right), \nabla\II_h\vphi \right)_{\omega_T} \\
&\quad + \left( D_0 \ssshk^+, \II_h\vphi \right)_{\Omega_T}
+ \left( D_0\left(\ssshk^+ \times \Pi_h\mmmhk^+\right), \II_h\vphi \right)_{\omega_T} \\
&\qquad = \beta \left(\Pi_h\mmmhk^+\otimes\jjj_k^+,\nabla\II_h\vphi \right)_{\omega_T}
- \beta \left(\jjj_k^+\cdot\nnn,\Pi_h\mmmhk^+\cdot\II_h\vphi \right)_{(0,T)\times(\partial\Omega\cap\partial\omega)},
\end{split}
\end{equation}
where $\II_h \vphi(t)=\II_h \left(\vphi(t)\right)$ for all $t \in (0,T)$.
Passing~\eqref{eq:last} to the limit for $(h,k) \to 0$, due to the convergence properties stated in Proposition~\ref{lem:subsequences}, in combination with the standard approximation properties of $\II_h$, we deduce
\begin{equation*}
\begin{split}
& \int_0^T \left\langle\partial_t \sss(t),\vphi(t) \right\rangle
+ \left( D_0 \nabla\sss,\nabla\vphi \right)_{\Omega_T}
- \beta\beta' \left( D_0 \mmm\otimes\left(\nabla\sss\cdot\mmm\right),\nabla\vphi \right)_{\omega_T} \\
& \quad + \left( D_0 \sss,\vphi \right)_{\Omega_T}
+ \left( D_0\left(\sss\times\mmm\right),\vphi \right)_{\omega_T}
= \beta \left(\mmm\otimes\jjj,\nabla\vphi \right)_{\omega_T}
- \beta \left(\jjj\cdot\nnn,\mmm\cdot\vphi\right)_{(0,T)\times(\partial\Omega\cap\partial\omega)}.
\end{split}
\end{equation*}
By density, this is also true for all $\vphi \in C^{\infty}(0,T;\HHH^1(\Omega))$.
Hence in particular for each $\zzeta\in\HHH^1(\Omega)$ and a.e.\ $t \in (0,T)$ we have~\eqref{eq:weak_spin_diff}.
\par Since $\mmmhk\subweakto\mmm$ in $\HHH^1(\omega_T)$ with $\mmmhk(0)=\mmmh^0$, and $\ssshk\subweakto\sss$ in $H^1(0,T;\wtd\HHH^{-1}(\Omega))$ with $\ssshk(0)=\sssh^0$, assumption~(H1) allows to deduce $\mmm(0)=\mmm^0$ and $\sss(0)=\sss^0$ in the sense of traces.
\end{proof}
%%%%%%%%%%%%%%%%%%%%%%%%%%%%%%%%%%%%%%%%%%%%%%%%%%%%%%%%%%%%%%%%%%%%%%%%%%%%%%%%%%%%%%%%%%%%%%%%%%%%
\section{Energy estimate}
\label{sec:energy}
%%%%%%%%%%%%%%%%%%%%%%%%%%%%%%%%%%%%%%%%%%%%%%%%%%%%%%%%%%%%%%%%%%%%%%%%%%%%%%%%%%%%%%%%%%%%%%%%%%%%
\noindent In this section, we exploit our constructive convergence proof to derive an energy estimate for weak solutions of SDLLG, which is also meaningful from a physical point of view.
The total magnetic Gibbs free energy from~\eqref{eq:ll_energy} is strongly related to the standard form~\eqref{eq:Llg} of LLG and does not take into account the interaction between the magnetization and spin accumulation.
As we are dealing with the augmented form~\eqref{eq:Spin_llg} of LLG, we extend~\eqref{eq:ll_energy} and define the free energy of the system by
\begin{equation} \label{eq:new_energy}
\EE (\MMM,\SSS)
= \frac{A}{\Ms^2} \int_{\omega}\abs{\nabla\MMM}^2
+ K \int_{\omega}\phi\left(\MMM/\Ms\right)
- \mu_0 \int_{\omega}\HHH_e\cdot\MMM
- \frac{\mu_0}{2} \int_{\omega} \HHH_s(\MMM)\cdot\MMM
-J \int_{\omega}\SSS\cdot\MMM.
\end{equation}
This definition is in agreement with~\eqref{eq:Spin_llg}, since in this case it is easy to see that
\begin{equation*}
- \frac{\delta \EE(\MMM,\SSS)}{\delta \MMM} = \mu_0 \HEFF(\MMM) + J \SSS,
\end{equation*}
where the effective field is given by~\eqref{eq:Eff_field}.
A simple formal computation shows that strong solutions to~\eqref{eq:Spin_llg} satisfy
\begin{equation} \label{eq:energy_decay}
\frac{d\EE}{dt}
= - \frac{\alpha}{\gamma \Ms} \int_{\omega} \abs{\MMMt}^2
- \mu_0 \int_{\omega} \frac{\partial\HHH_e}{\partial t}\cdot\MMM
- J \int_{\omega} \frac{\partial\SSS}{\partial t}\cdot\MMM.
\end{equation}
Neglecting the spin accumulation term and assuming that $\HHH_e$ is constant in time, equation~\eqref{eq:energy_decay} reduces to
\begin{equation*}
\frac{d\EE}{dt}
= - \frac{\alpha}{\gamma \Ms} \int_{\omega} \abs{\MMMt}^2 \leq 0,
\end{equation*}
which reveals the well-known dissipative behavior of solutions to the standard form~\eqref{eq:Llg} of LLG.
\par The main aim of this section is to prove a property corresponding to~\eqref{eq:energy_decay} in the context of weak solutions.
To this end, we move to the nondimensional framework introduced in Section~\ref{sec:problem} and consider the following assumptions:
\begin{itemize}
\item[(H4)] the operator $\ppi:\LLL^2(\omega)\to\LLL^2(\omega)$ from~\eqref{eq:gen_eff_field} is linear, self-adjoint, and bounded;
\item[(H5)] it holds
\begin{equation*}
\ppih(\www_{hk}) \to \ppi(\www) \text{ in } \LLL^2(\omega_T) \quad \text{ as } (h,k) \to 0
\end{equation*}
for any sequence $\www_{hk} \to \www$ in $\LLL^2(\omega_T)$, which is slightly stronger than~(H3);
\item[(H6)] the applied field $\fff$ belongs to $H^1(0,T;\LLL^2(\omega))$.
\end{itemize}
%%%%%%%%%%%%%%%%%%%%%%%%%%%%%%%%%%%%%%%%%%%%%%%%%%%%%%%%%%%%%%%%%%%%%%%%%%%%%%%%%%%%%%%%%%%%%%%%%%%%
\begin{remark}
For some fixed easy axis $\eee\in\mathbb{S}^2$ and the corresponding uniaxial anisotropy density function $\phi(\mmm)=1-(\mmm\cdot\eee)^2$, the operator
\begin{equation} \label{eq:practical_pi}
\ppi(\mmm)
= \hhh_s(\mmm)-\Cani\nabla\phi(\mmm)
= \hhh_s(\mmm) + 2 \Cani (\eee\cdot\mmm)\mmm
\end{equation}
satisfies~{\rm{(H4)}}.
Moreover, all the stray field discretizations mentioned in~Remark~\ref{rem:discrete_pi} satisfy~{\rm{(H5)}}, see~\cite{bffgpprs2014}.
The operator $\ppi$ from~\eqref{eq:practical_pi} is even well defined and bounded as operator $\ppi:\LLL^p(\omega)\to\LLL^p(\omega)$ for all $1<p<\infty$, see~\cite{praetorius2004}.
Unlike~\cite{akt2012}, the proof of our energy estimate, see Theorem~\ref{thm:energy} below, avoids this additional regularity, but only relies on the energy setting $p=2$.
\end{remark}
%%%%%%%%%%%%%%%%%%%%%%%%%%%%%%%%%%%%%%%%%%%%%%%%%%%%%%%%%%%%%%%%%%%%%%%%%%%%%%%%%%%%%%%%%%%%%%%%%%%%
\noindent Up to an additive constant, the nondimensional counterpart of~\eqref{eq:new_energy} reads
\begin{equation} \label{eq:energy_adimensional}
\EE(\mmm,\sss)
= \frac{1}{2}\Ce\int_{\omega}\abs{\nabla\mmm}^2
- \int_{\omega}\fff\cdot\mmm
- \frac{1}{2} \int_{\omega} \ppi(\mmm)\cdot\mmm
- c \int_{\omega}\sss\cdot\mmm.
\end{equation}
The following theorem proves an energy estimate which generalizes~\eqref{eq:energy_decay} to weak solutions.
%%%%%%%%%%%%%%%%%%%%%%%%%%%%%%%%%%%%%%%%%%%%%%%%%%%%%%%%%%%%%%%%%%%%%%%%%%%%%%%%%%%%%%%%%%%%%%%%%%%%
\begin{theorem} \label{thm:energy}
Suppose that assumptions~{\rm{(H1)--(H2)}} and~{\rm{(H4)--(H6)}} are satisfied.
Let $(\mmm,\sss)$ be a weak solution of SDLLG obtained as a weak limit of the finite element solutions from Algorithm~\ref{alg} for $1/2 < \theta \leq 1$.
Then, the energy functional from~\eqref{eq:energy_adimensional} satisfies
\begin{equation} \label{eq:weak_energy_decay}
\begin{split}
& \EE(\mmm(t),\sss(t))
+ \alpha \int_{t'=0}^t \norm{\mmmt(t')}{\LLL^2(\omega)}^2
+ \int_{t'=0}^t \left(\partial_t\fff(t'),\mmm(t') \right)_{\omega}
+ c \int_{t'=0}^t \left\langle\ssst(t'),\mmm(t') \right\rangle \\
& \quad \leq \EE(\mmm_0,\sss_0)
\end{split}
\end{equation}
for almost all $t \in (0,T)$.
\end{theorem}
%%%%%%%%%%%%%%%%%%%%%%%%%%%%%%%%%%%%%%%%%%%%%%%%%%%%%%%%%%%%%%%%%%%%%%%%%%%%%%%%%%%%%%%%%%%%%%%%%%%%
\begin{proof}
Given $t \in (0,T)$, let $0 \leq j \leq N-1$ such that $t \in [t_j,t_{j+1})$.
Let $0 \leq i \leq j$.
From Lemma~\ref{lem:new_stability}, we get
\begin{equation*}
\begin{split}
& \EE(\mmmh^{i+1},\sssh^{i+1})
- \EE(\mmmh^i,\sssh^i) \\
& \quad = \frac{1}{2}\Ce\left(\norm{\nabla\mmmh^{i+1}}{\LLL^2(\omega)}^2 - \norm{\nabla\mmmh^i}{\LLL^2(\omega)}^2 \right)
- \left(\fff^{i+1},\mmmh^{i+1}\right)_{\omega}
+ \left(\fff^i,\mmmh^i\right)_{\omega} \\
& \qquad - \frac{1}{2}\left(\ppi(\mmmh^{i+1}),\mmmh^{i+1}\right)_{\omega}
+ \frac{1}{2}\left(\ppi(\mmmh^i),\mmmh^i\right)_{\omega}
- c \left(\sssh^{i+1},\mmmh^{i+1}\right)_{\omega}
+ c \left(\sssh^i,\mmmh^i\right)_{\omega} \\
& \quad = -\alpha k \norm{\vvvh^i}{\LLL^2(\omega)}^2
- \Ce k^2 \left(\theta-\frac{1}{2} \right)\norm{\nabla\vvvh^i}{\LLL^2(\omega)}^2
\underbrace{- \left(\fff^{i+1},\mmmh^{i+1}\right)_{\omega}
+ \left(\fff^i,\mmmh^i + k \vvvh^i\right)_{\omega}}_{= \ T_1} \\
& \qquad \underbrace{+ k \left(\ppih(\mmmh^i),\vvvh^i\right)_{\omega}
- \frac{1}{2}\left(\ppi(\mmmh^{i+1}),\mmmh^{i+1}\right)_{\omega}
+ \frac{1}{2}\left(\ppi(\mmmh^i),\mmmh^i\right)_{\omega} }_{= \ T_2} \\
& \qquad \underbrace{- c \left(\sssh^{i+1},\mmmh^{i+1}\right)_{\omega}
+ c \left(\sssh^i,\mmmh^i + k\vvvh^i \right)_{\omega}}_{= \ T_3}.
\end{split}
\end{equation*}
By definition~\eqref{eq:alg:2}, it holds $\mmmh^{i+1} = \mmmh^i + k \vvvh^i$.
We thus obtain
\begin{equation*}
T_1
= - \left(\fff^{i+1}-\fff^i,\mmmh^{i+1}\right)_{\omega}
- \left(\fff^i, \mmmh^{i+1} - \mmmh^i - k \vvvh^i \right)_{\omega}
= - k \left(\dt\fff^{i+1}, \mmmh^{i+1} \right)_{\omega}.
\end{equation*}
Analogously, we see that $T_3 = - ck \left(\dt\sssh^{i+1},\mmmh^{i+1}\right)_{\omega}$.
Since $\ppi$ is linear and self-adjoint, \eqref{eq:alg:2} also reveals
\begin{equation*}
\begin{split}
T_2
& = - k \left(\ppi(\mmmh^i)-\ppih(\mmmh^i),\vvvh^i\right)_{\omega}
+ k \left(\ppi(\mmmh^i),\vvvh^i\right)_{\omega} \\
& \quad - \frac{1}{2} \left(\ppi(\mmmh^{i+1})-\ppi(\mmmh^i),\mmmh^{i+1}\right)_{\omega}
- \frac{1}{2} \left(\ppi(\mmmh^i),\mmmh^{i+1} - \mmmh^i \right)_{\omega} \\
& = - k \left(\ppi(\mmmh^i)-\ppih(\mmmh^i),\vvvh^i\right)_{\omega}
+ \frac{1}{2} k \left(\ppi(\mmmh^i),\vvvh^i\right)_{\omega}
- \frac{1}{2} k \left(\ppi(\vvvh^i),\mmmh^{i+1}\right)_{\omega} \\
& = - k \left(\ppi(\mmmh^i)-\ppih(\mmmh^i),\vvvh^i\right)_{\omega}
- \frac{1}{2} k^2 \left(\ppi(\vvvh^i),\vvvh^i\right)_{\omega}.
\end{split}
\end{equation*}
Altogether, we thus obtain
\begin{equation*}
\begin{split}
& \EE(\mmmh^{i+1},\sssh^{i+1})
- \EE(\mmmh^i,\sssh^i)
+ \alpha k \norm{\vvvh^i}{\LLL^2(\omega)}^2
+ k \left(\dt\fff^{i+1}, \mmmh^{i+1} \right)_{\omega}
+ ck \left(\dt\sssh^{i+1},\mmmh^{i+1}\right)_{\omega} \\
& \quad = - \Ce k^2 \left(\theta-\frac{1}{2} \right)\norm{\nabla\vvvh^i}{\LLL^2(\omega)}^2
- k \left(\ppi(\mmmh^i)-\ppih(\mmmh^i),\vvvh^i\right)_{\omega}
- \frac{1}{2} k^2 \left(\ppi(\vvvh^i),\vvvh^i\right)_{\omega}.
\end{split}
\end{equation*}
Since $\ppi$ is a bounded operator, it follows that
\begin{equation*}
\begin{split}
& \EE(\mmmh^{i+1},\sssh^{i+1})
- \EE(\mmmh^i,\sssh^i)
+ \alpha k \norm{\vvvh^i}{\LLL^2(\omega)}^2
+ k \left(\dt\fff^{i+1}, \mmmh^{i+1} \right)_{\omega} \\
& \quad + ck \left(\dt\sssh^{i+1},\mmmh^{i+1}\right)_{\omega}
+ k \left(\ppi(\mmmh^i)-\ppih(\mmmh^i),\vvvh^i\right)_{\omega}
\lesssim k^2 \norm{\vvvh^i}{\LLL^2(\omega)}^2.
\end{split}
\end{equation*}
Summation over $0 \leq i \leq j$ and the boundedness from Proposition~\ref{lem:discrete_stability2} yield
\begin{equation*}
\begin{split}
& \EE(\mmmhk^+(t_{j+1}),\ssshk^+(t_{j+1}))
- \EE(\mmmh^0,\sssh^0)
+ \alpha \int_{t'=0}^{t_{j+1}}\norm{\vvvhk^-(t')}{\LLL^2(\omega)}^2
+ \int_{t'=0}^{t_{j+1}} \left(\partial_t\fff_k(t'),\mmmhk^+(t')\right)_{\omega} \\
& \quad + c \int_{t'=0}^{t_{j+1}} \left\langle\partial_t\ssshk(t'),\mmmh^+(t')\right\rangle
+ \int_{t'=0}^{t_{j+1}}\left(\ppi(\mmmhk^-(t'))-\ppih(\mmmhk^-(t')),\vvvhk^-(t')\right)_{\omega} \\
& \qquad \lesssim k \int_{t'=0}^{t_{j+1}} \norm{\vvvhk^-(t')}{\LLL^2(\omega)}^2 \lesssim k.
\end{split}
\end{equation*}
The available convergence results on $\mmmhk^{\pm}$, $\ssshk^+$, $\ssshk$, $\vvvhk^-$, and $\fff_k$, as well as assumption~(H5), allow us to employ standard arguments with weakly lower semicontinuity for the limit $(h,k)\to0$. This concludes the proof of~\eqref{eq:weak_energy_decay}.
\end{proof}
%%%%%%%%%%%%%%%%%%%%%%%%%%%%%%%%%%%%%%%%%%%%%%%%%%%%%%%%%%%%%%%%%%%%%%%%%%%%%%%%%%%%%%%%%%%%%%%%%%%%
\section*{Acknowledgments}
%%%%%%%%%%%%%%%%%%%%%%%%%%%%%%%%%%%%%%%%%%%%%%%%%%%%%%%%%%%%%%%%%%%%%%%%%%%%%%%%%%%%%%%%%%%%%%%%%%%%
\noindent The research of GH, MP, DP, and DS is supported through the project MA09-029 funded by Vienna Science and Technology Fund (WWTF).
The research of DP and MR is supported through the doctoral school \emph{Dissipation and Dispersion in Nonlinear PDEs}, funded by the Austrian Science Fund (FWF) under grant W1245, and through the innovative projects initiative of Vienna University of Technology.
Additionally, CA and DS acknowledge financial support of the Austrian Federal Ministry of Economy, Family and Youth and the National Foundation for Research, Technology and Development (Christian Doppler Laboratory \emph{Advanced Magnetic Sensing and Materials}).
GH would like to thank the EPSRC (grant EP/K008412/1) and the Royal Society (UF080837) for financial support.
%%%%%%%%%%%%%%%%%%%%%%%%%%%%%%%%%%%%%%%%%%%%%%%%%%%%%%%%%%%%%%%%%%%%%%%%%%%%%%%%%%%%%%%%%%%%%%%%%%%%
\bibliographystyle{model1b-num-names}
\bibliography{ref/refBooks,ref/refPapers,ref/refPhysicsPapers,ref/refUnpub}
%%%%%%%%%%%%%%%%%%%%%%%%%%%%%%%%%%%%%%%%%%%%%%%%%%%%%%%%%%%%%%%%%%%%%%%%%%%%%%%%%%%%%%%%%%%%%%%%%%%%
\end{document}